\documentclass[a4paper]{article}
\usepackage[english]{babel}
\usepackage[latin1]{inputenc}
\usepackage{dsfont}
\usepackage{graphicx}
\usepackage{color}
\usepackage{epsfig}
\usepackage{amsthm,amssymb,amsbsy,amsmath,amsfonts,amssymb,amscd,mathrsfs}
\usepackage{colortbl}
\usepackage{verbatim}
\usepackage{url}
\usepackage{enumerate}
\usepackage{bbm}
\usepackage[multiple]{footmisc}
\newtheorem{rem}{Remark}[section]
\newtheorem{prop}{Proposition}[section]
\newtheorem{defi}{Definition}[section]
\newtheorem{lem}{Lemma}[section]

\newtheorem{thm}{Theorem}[section]

\def\ds{\displaystyle}

\newcommand{\dd}{\mathrm{d}}
\newcommand{\eps}{\varepsilon}
\newcommand{\R}{\mathbb{R}}

\newcommand\be{\begin{equation}}
\newcommand\ee{\end{equation}}
\numberwithin{equation}{section}

\title{Second-order analysis for the time crisis problem\footnote{This research benefited from the support of the LABEX NUMEV Montpellier and from the ESI (University of Vienna). }}
\author{Terence Bayen\footnote{IMAG, Univ Montpellier, CNRS, Montpellier, France  {\tt\small
    terence.bayen@umontpellier.fr}} ,
   Laurent Pfeiffer\footnote{Institute of Mathematics, University of Graz, Austria {\tt\small laurent.pfeiffer@uni-graz.at}
}}
\date{\today}

\begin{document}

\maketitle

\begin{abstract}
In this article, we prove second-order necessary optimality conditions for the so-called time crisis problem that comes up within the context of viability theory. 
It consists in minimizing the time spent by solutions of a controlled dynamics 
outside a given subset $K$ of the state space. One essential feature is the discontinuity of the characteristic function 
involved in the cost functional. Thanks to a change of time and an augmentation of the dynamics, 
we relate the time crisis problem to an auxiliary Mayer control problem. 
This allows us to use the classical tools of optimal control for obtaining optimality conditions. 
Going back to the original problem, we deduce that way second order optimality conditions for the time crisis problem. 
\end{abstract}

\medskip

 
\section{Introduction}

Given a controlled dynamics $f: \R^n \times \R^m \rightarrow \R^n$ with associated system
\be{\label{sys1-intro}}
\dot{x}=f(x,u),
\ee
and given a non-empty closed subset $K\subset \R^n$,
the {\it{time crisis problem}} 
amounts to minimize the time spent by solutions of \eqref{sys1-intro} outside the set $K$ w.r.t.~admissible controls $u$:
\be{\label{TC1-intro}}
\inf_{u(\cdot)} \int_0^{T} \mathds{1}_{K^c}(x_u(t,x_0)) \ \dd t.  \tag{TC}
\ee
Here $T\in \R_+^*\cup \{+\infty\}$ and $\mathds{1}_{K^c}$ denotes the 
characteristic function of the complementary of $K$ in $\R^n$:
$$
\mathds{1}_{K^c}(x):=
\begin{cases}
\begin{array}{lll}
0 & \mathrm{if}&  x\in K,\\
1 & \mathrm{if}& x\notin K. 
\end{array}
\end{cases}
$$
In addition, $x_u(\cdot,x_0)$ denotes a solution of \eqref{sys1-intro} such that $x_u(0,x_0)=x_0$ with $x_0\in \R^n$. 
Originally, the time crisis problem was introduced in \cite{DSP} in the context of viability theory \cite{aubin1,ABSP} with $T=+\infty$.  
The value of the infimum in \eqref{TC1-intro} (possibly equal to $+\infty$) is 
the so-called {\it{minimal time crisis function}} and it can be written $\theta(x_0)$ as an explicit function of the initial condition $x_0$.  
When $x_0$ belongs to the domain of $\theta$, this function measures the minimal time spent by solutions of the system outside the set $K$,  
which models state constraints. Finding an optimal control in \eqref{TC1-intro} allows to obtain  
significant informations on the system (in terms of violation of state constraints) in several application models (see, {\it{e.g.}}, \cite{bayen3}). 

With regard to the properties satisfied by the minimal time crisis function, there are
two essential features: first, the integrand is discontinuous at every time $t$ at which $x_u(\cdot,x_0)$ crosses the boundary of $K$. 
Second, the functional may involve an infinite horizon, which also requires a careful attention. 
So, one cannot directly apply the classical necessary optimality conditions \cite{Pontry,vinter} to find optimal controls. 
In \cite{DSP}, sufficient optimality conditions have been derived based on the characterization of the value function as a generalized solution of an Hamilton-Jacobi equation. In \cite{bayen2,bayen1}, first-order optimality conditions were given 
thanks to the hybrid maximum principle that is an extension of Pontryagin's Principle \cite{clarke2013} (see also \cite{piccoli,trelat}).   
Note that the obtention of such conditions relies on a transversality assumption on optimal trajectories in order to properly define extremals of the problem (see also \cite{trelat}). This assumption means that at each crossing time of the set $K$, a trajectory does not hit the boundary tangentially.   

Our aim in this paper is to go one one step further and to provide second-order optimality conditions 
for the time crisis problem when $T<+\infty$ and with an additional terminal-payoff (see Problem \eqref{TC1} hereafter).    
Doing so, we  introduce a time re-parametrization and an augmented controlled system 
(based on an explicit description of $K$ and of the admissible control set) 
that allow us to transform \eqref{TC1} into a classical Mayer  
control problem ($\mathrm{P}$) with mixed initial-terminal constraints, for which we apply the usual tools of optimal control (first- and second-order optimality conditions). 
The above transformation is in the spirit of \cite{dmitruk1,dmitruk2,dmitruk3} (it is used in  \cite{dmitruk1} to relate hybrid control problems 
to classical control problems for which one can apply first-order optimality conditions).   
It is made possible assuming that a nominal optimal trajectory possesses a finite number of transverse crossing times. 
This assumption implies that small perturbations of the nominal trajectory necessarily 
have the same number of crossing times (as the nominal one), which is a key property for the transformation of the problem. 
We also impose an inward pointing condition on the control constraint.  

The paper is organized as follows: Section \ref{formul_sec} introduces the main assumptions and Section \ref{transfo-sec} the transformation of the time crisis problem into a Mayer control problem ($\mathrm{P}$).
In Section \ref{NOC-sec}, we prove first- and second-order necessary optimality conditions for the transformed ($\mathrm{P}$), which are then translated into optimality conditions for the time crisis problem. In a first step, these results are formulated with generalized Lagrange multipliers. In a second step, we show that they are still valid when restricted to Pontryagin multipliers. We also prove that Pontryagin multipliers are non-singular and unique, up to a multiplicative constant.
The results obtained in Section \ref{transfo-sec} and Section \ref{NOC-sec} are given for the situation of a single crossing time. They can be naturally extended to the situation with several crossing times, as explained then in Section \ref{NOC_sec_gen}.

\section{Formulation of the problem and assumptions} \label{formul_sec}

Throughout the rest of the paper $T>0$ is fixed, $n$, $m$, and $l$ are positive integers and $|\cdot|$ stands for the euclidean norm in 
$\R^s$ associated with the standard inner product written $a\cdot b$ for $a,b\in \R^s$ ($s$ being a positive integer).
Given a non-empty closed subset $K$ of $\R^n$, we denote by $\mathrm{Int}(K)$, $\partial K$, and $K^c$ the interior, the boundary, and the complementary of the set $K$. 
In the sequel, we consider an autonomous controlled dynamics $f:\R^n \times \R^m \rightarrow \R^n$ whose associated system is
\be{\label{sys}}
\dot{x}=f(x,u),
\ee
where $x$ is the state and $u$ is a measurable control with values in a non-empty closed subset $U$ of $\R^m$. 
We suppose that the dynamics fulfills the following (standard) assumptions: 
\begin{itemize}
\item The mapping $f$ is of class $C^2$ w.r.t.~$(x,u)$, and satisfies the linear growth condition: there exist $c_1>0$ and $c_2>0$ such that for all $x\in \R^n$ and all $u\in U$, one has:
\be{\label{croissance}}
|f(x,u)|\leq c_1 |x|+c_2. 
\ee
\item For any $x\in \R^n$, the velocity set $F(x):=\{f(x,u)\; ; \; u\in U\}$ is a non-empty compact convex subset of $\R^n$.
\end{itemize}
Under these assumptions, for any $x_0\in \R^n$, there is a unique solution $x_u(\cdot)$ of the Cauchy problem
\be{\label{cauchy-pbm1}}
\left\{
\begin{array}{cl}
\dot{x}&=f(x,u),\\
x(0)&=x_0,
\end{array}
\right.
\ee
defined over $[0,T]$. Let now $K$ be a closed subset of $\R^n$ with non-empty interior, 
and let $\phi:\R^n\rightarrow \R$ be a terminal pay-off (of class $C^2$). Our aim in this paper is to investigate necessary optimality conditions 
for the optimal control problem:
\be{\label{TC1}}
\inf_{u \in \mathcal{U}} J_T(u):=\phi(x_u(T))+ \int_0^{T} \mathds{1}_{K^c}(x_u(t)) \ \dd t,
\ee
where the set of admissible controls $\mathcal{U}$ is given by
$$
\mathcal{U}:=\{u:[0,T] \rightarrow U \; ; \; u \; \mathrm{meas}. \}. 
$$
By an optimal solution of \eqref{TC1}, we mean a (global) optimal control $u \in \mathcal{U}$ of \eqref{TC1}. 
Existence of an optimal solution for \eqref{TC1} is standard (we refer to \cite{bayen1,DSP}). 
Note that when $\phi \equiv 0$, we retrieve the so-called {\it{time crisis problem}} over $[0,T]$ as in \cite{bayen1}.

To express optimality conditions, it is convenient to  write $U$ and $K$ as sub-level sets of given functions satisfying qualification conditions. We therefore need to introduce additional assumptions. We fix for the rest of the article a solution $\bar{u} \in \mathcal{U}$ to \eqref{TC1}, with associated trajectory $\bar{x}:= x_{\bar{u}}$, satisfying Assumption (H1).

\begin{enumerate}[(H1)]
\item There is a function $c:\R^m \rightarrow \R^l$ of class $C^2$ such that
\be{\label{control-set}}
U= \{u \in \R^m \; ; \; c_i(u) \leq 0, \; 1 \leq i \leq l\}.
\ee
For $\delta > 0$ and $i \in \{ 1,..., l \}$, we define $\Delta_{c,i}^{\delta}:= \{ t \in (0,T) \,;\, c_i(\bar{u}(t)) \leq - \delta \}$ and
for $\delta > 0$ and $t \in (0,T)$, we define
\begin{equation*}
I_c^\delta(t):= \{ i \in \{ 1,...,l \} \,;\, t \in \Delta_{c,i}^\delta \}.
\end{equation*}
Given a subset $J= \{ i_1,..., i_{|J|} \} \subseteq \{ 1,...,l \}$ of cardinality $|J|$, we set $c_{J}(u):= (c_{i_1}(u),...,c_{i_{|J|}}(u)) \in \R^{|J|}$. We assume that there exist $\varepsilon > 0$ and $\delta > 0$ such that
\begin{equation} \label{LIG}
\varepsilon | \xi | \leq \nabla c_{I_c^\delta(t)}(\bar{u}(t)) \xi, \quad \forall \xi \in \R^{|I_c^\delta(t)|}, \text{ for a.e.\@ $t \in (0,T)$}.
\end{equation}
\end{enumerate}
Throughout the article, we also assume that $K$ satisfies the following hypothesis. 
\begin{enumerate}[(H2)]
\item There is a function $g: \R^n \rightarrow \R$ of class $C^1$ such that 
\be
K=\{ x \in \R^n \; ; \; g(x) \leq 0\}. 
\ee
\end{enumerate}

\begin{rem}
Inequality \eqref{LIG}, referred to as \emph{linear independence of gradients of active constraints} condition is classical. The reader can easily check that it implies the following properties (see, {\it{e.g.}}, \rm{\cite{BH09}}):
\begin{itemize}
\item Inward pointing condition: there exist $\varepsilon > 0$ and $v \in L^\infty(0,T;\R^m)$ such that 
\be{\label{IPC}}
c(\bar{u}(t)) + Dc(\bar{u}(t))v(t) \leq -\varepsilon \quad \mathrm{for\ a.e.} \; t \in (0,T). 
\ee
\item There exists $\delta>0$ such that the following mapping is onto:
\begin{equation} \label{qualif2}
v \in L^2(0,T;\R^m) \mapsto \Big( (Dc_i(\bar{u}(\cdot)) v(\cdot))_{|\Delta_{c,i}^\delta} \Big)_{i=1,...,l} \in \prod_{i=1}^l L^2(\Delta_{c,i}^\delta).
\end{equation}
Note that the inward pointing condition ensures the existence of a Lagrange multiplier in $L^\infty(0,T;\R^l)$ for 
\eqref{TC1} under the control constraint $c(u)\leq 0$ (see also {\rm \cite{BDP14,BH09,BO10}}).
\end{itemize}
\end{rem}

\begin{rem}
For the results dealing with first-order optimality conditions in Subsection \ref{NOC1-sec}, it is enough to assume that $f$, $\phi$, $c$, and $g$ are of class $C^1$, and the condition \eqref{LIG} can be replaced by the inward pointing condition, which is weaker.
\end{rem}

The analysis of optimal controls of \eqref{TC1} and associated trajectories relies on the notion of {\it{crossing time}} that we now recall. 
\begin{defi}
{\rm{(i)}} A crossing time from $K$ to $K^c$  
is a time $t_c \in (0,T)$ for which there is $\eps>0$ such that for any time 
$t\in (t_c-\eps,t_c]$ (resp. $t\in (t_c,t_c+\eps)$) one has $\bar{x}(t)\in K$ (resp. $\bar{x}(t) \in K^c$). 
\\
{\rm{(ii)}} A crossing time $t_c$ from $K$ to $K^c$ is  {\it{transverse}} if the control $\bar{u}$ is right- and left- continuous at time $t_c$, and if 
\be{\label{cond-transverse}}
\dot{\bar{x}}(t_c^\pm) \cdot \nabla g(\bar{x}(t_c))\not=0. 
\ee
\end{defi}

When \eqref{cond-transverse} is fulfilled, the trajectory $\bar{x}$ does not hit the boundary of $K$ tangentially at time $t_c$. Note that there are similar definitions for crossing times from $K^c$ to $K$ and transverse crossing times from $K^c$ to $K$. 
The analysis that we carry out in this paper relies on the following assumption on $\bar{x}$:
\begin{enumerate}[(H3)]
\item  The optimal trajectory $\bar{x}$ possesses 
exactly $r \in \mathbb{N}^*$ transverse crossing times $\bar{\tau}_1 < \cdots < \bar{\tau}_r$ in $(0,T)$ such that $\bar{\tau}_{2i+1}$ (resp.\@ $\bar{\tau}_{2i}$) is a crossing time from $K$ to $K^c$ (resp.\@ from $K^c$ to $K$). For all $t \in [0,T] \backslash \{ \bar{\tau}_1,\cdots,\bar{\tau}_r \}$, $g(\bar{x}(t)) \neq 0$.
\end{enumerate}
Assumption (H3) implicitly supposes that the initial condition satisfies $x_0 \in \mathrm{Int}(K)$, but we could consider as well $x_0$ in $K^c$ 
with slight modifications. It also excludes the chattering phenomenon (see \cite{zelikin}) at the boundary of the set $K$, that is, we do not consider in this study optimal trajectories that could eventually switch an infinite number of times at the boundary of $K$ over a finite horizon (see also \cite{trelat2}).  

\section{Reformulation of the time crisis problem}{\label{transfo-sec}}

We recall that $\bar{u}$ is a fixed solution to Problem (\ref{TC1}) with associated trajectory $\bar{x} = x_{\bar{u}}$, satisfying Assumption (H1). We moreover assume that (H3) is satisfied with $r=1$; the unique crossing time is denoted by $\bar{\tau}$.
The goal of this section is to provide a formulation of \eqref{TC1} as a classical optimal control problem.
This will go in two steps: 
first, a change of time in \eqref{sys} is introduced (Section \ref{time-transfo-sec}). Second, 
we consider an augmented system associated with the dynamics obtained after the first transformation (Section \ref{augment-transfo-sec}). 

\subsection{Time transformation}{\label{time-transfo-sec}}

We start by introducing a time transformation as follows. For $\tau \in (0,T)$, 
let $\pi_\tau:[0,2]\rightarrow [0,T]$, $s\mapsto t:=\pi_\tau(s)$ be the piecewise-affine function defined as
\be{\label{time-change1}}
\pi_\tau(s):=
\left\{
\begin{array}{lll}
\tau s, & \mathrm{if}&s\in [0,1],\\
(T-\tau)s+2\tau-T, & \mathrm{if}&s\in [1,2].
\end{array}
\right.
\ee
It is easily seen that the change of variable $\pi_\tau$ is one-to-one if and only if $\tau\in (0,T)$. 
Now, given $u \in \mathcal{U}$, we set for $s \in [0,2]$
\be{\label{change1}}
\left|
\begin{array}{cl}
\tilde u(s)&:=u(\pi_\tau(s)), \\
\tilde x(s)&:=x(\pi_\tau(s)),
\end{array}
\right.
\ee
where $x$ denotes the unique solution of \eqref{cauchy-pbm1} associated with $u$. The trajectory $\tilde{x}$ is then the unique solution to the differential equation
\be{\label{cauchy-pbm3}}
\left\{
\begin{array}{rl}
\ds \frac{d\tilde x}{ds}(s)&=\ds \frac{d\pi_{\tau}}{ds}(s) f(\tilde x(s),\tilde u(s)) \quad \mathrm{for\ a.e.} \; s\in [0,2],\vspace{0.15cm}\\
\tilde x(0)&=x_0. 
\end{array}
\right.
\ee
We can consider now the following set of admissible controls
$$
\mathcal{\tilde U}:=\{\tilde u :[0,2] \rightarrow U \; ; \; \tilde u \; \mathrm{meas}.\}, 
$$
and the following optimal control problem:
\be{\label{TC3}}
\inf_{ \tilde u \in  \mathcal{\tilde U}, \; \tau\in (0,T)} \phi(\tilde x_{\tilde u,\tau}(2))+T-\tau \quad \mathrm{s.t.} \; g(\tilde x_{\tilde u,\tau}(1))=0, 
\ee
where $\tilde x_{\tilde u,\tau}$ is the unique solution of \eqref{cauchy-pbm3}.
Let us emphasize the fact that $\tau$ is an optimization variable of the problem, involved in the dynamics of the system. The crossing time of the trajectory is fixed to 1.
We adopt the following definition of minimum. 

\begin{defi}{\label{min-faible1}}
A pair $(\tilde u,\tau)\in \mathcal{\tilde U} \times (0,T)$ is a weak minimum of \eqref{TC3} if there exists $\eps>0$ such that for all control 
$\tilde u' \in  \mathcal{\tilde U}$ and all $\tau'\in (0,T)$ one has:
\be{\label{def-weak-min}}
\|\tilde u'-\tilde u\|_{L^\infty(0,2;\R^m)} \leq \eps \quad \mathrm{and} \quad |\tau-\tau'| \leq \eps \quad \Rightarrow \quad \phi(\tilde x_{\tilde u,\tau}(2))+T-\tau \leq 
\phi(\tilde x_{\tilde u',\tau'}(2))+T-\tau',
\ee
where $\tilde x_{\tilde u',\tau'}$ is the unique solution of \eqref{cauchy-pbm3} associated with $\tilde u'$ and $\tau'$.
\end{defi}

The next proposition is a key result to reformulate \eqref{TC1} as a classical optimal control problem.
 
\begin{prop}{\label{minGminF}} 
Let $\tilde{u}:= \bar{u} \circ \pi_{\bar{\tau}}$. Then, $(\tilde u, \bar{\tau})$ is a weak minimum of \eqref{TC3}. 
\end{prop}

\begin{proof}
Let us set $\tilde{x}:= \tilde{x}_{\tilde{u},\bar{\tau}}= \bar{x} \circ \pi_{\bar{\tau}}$.
First, we show that there is $\eps_1>0$ such that for all $(\tilde u',\tau') \in  L^\infty(0,2;\R^m) \times (0,T)$ one has
\be{\label{tmp1-lemme-laurent}}
\left\{
\begin{array}{ll}
\|\tilde u - \tilde u'\|_{L^\infty(0,2;\R^m)} &\leq \eps_1,\\
|\bar{\tau}-\tau'|&\leq \eps_1,\\
g(\tilde x'(1))&=0,
\end{array}
\right.
\quad \Rightarrow \quad 
\forall s \in [0,1), \; g(\tilde x'(s))<0,
\ee
with $\tilde{x}' := \tilde{x}_{\tilde{u}',\tau'}$. 
Since {\rm{(H3)}} is satisfied, we have $\zeta:= \tau \nabla g(\tilde x(1))\cdot f(\tilde x(1),\tilde u(1^-))>0$ where 
$\tilde u(1^-):=\lim_{t \uparrow \bar{\tau}} \bar{u}(t)$. 
By continuity of $f$ and $g$ there is $\eta_1>0$ such that for any $(x,u,\tau)\in \R^n \times U \times (0,T)$ one has:
\begin{equation} \label{3.65}
\left\{
\begin{array}{ll}
|x-\tilde x(1)| & \leq \eta_1,\\
|u - \tilde u(1^-)| &\leq \eta_1,\\
|{\tau}-\bar{\tau}| & \leq \eta_1,
\end{array}
\right.
\quad \Rightarrow \quad 
\tau \nabla g(x)\cdot f(x,u) \geq \frac{\zeta}{2}>0.
\end{equation}
Now, by continuity of the trajectory and the control at the crossing time, there exists $\eta_2>0$ such that 
\be{\label{tmp-ineg2}}
|\tilde x (s)-\tilde x(1)| \leq \frac{\eta_1}{2} \quad \mathrm{and} \quad 
|\tilde u(s)-\tilde u(1^-)| \leq \frac{\eta_1}{2}, \quad
\text{for a.e.\@ $s \in (1-\eta_2,1)$}. 
\ee
Since $\tilde x(\cdot)$ is with values in the interior of $K$ over $[0,1-\eta_2]$, there is $\eta_3>0$ such that 
$$
\forall s \in [0,1-\eta_2], \quad g(\tilde x(s))\leq -\eta_3<0. 
$$
Recall now that the mapping 
\begin{equation} \label{eq:control_to_state_map}
(\tilde u',\tau') \in \mathcal{\tilde U} \times (0,T) \mapsto \tilde{x}_{\tilde u',\tau'} \in  L^\infty(0,2;\R^m)
\end{equation}
is continuous, when $\mathcal{\tilde{U}}$ is equipped with the $L^\infty$-norm. Hence, there exists $\eta_4>0$ such that 
\begin{equation} \label{eq:def_eta_4}
\left\{
\begin{array}{ll}
\|\tilde u - \tilde u' \|_{L^\infty(0,2;\R^m)} &\leq \eta_4,\\
|\bar{\tau}-\tau'|&\leq \eta_4,
\end{array}
\right.
\; \Rightarrow \quad 
\begin{array}{l}
\forall s\in [0,1-\eta_2], \; |g(\tilde x'(s))-g(\tilde x(s)) | < \eta_3, \\
\forall s\in [1-\eta_2,1], \; |\tilde x(s)-\tilde x'(s)|\leq \frac{\eta_1}{2}, 
\end{array}
\end{equation}
implying in particular that $s\mapsto g(\tilde x'(s))$ is negative for $s\in [0,1-\eta_2]$. 
It remains now to prove that $s\mapsto g(\tilde x'(s))$ is also negative for $s \in [1-\eta_2,1)$. 
Reducing $\eta_4$ if necessary, we may assume that $\eta_4 \leq \frac{\eta_1}{2}$. We now have
\begin{align*}
|\tilde x'(s)-\tilde x(1)| \leq |\tilde x'(s)-\tilde x(s)| +
|\tilde x(s)-\tilde x(1)| \leq \frac{\eta_1}{2} + \frac{\eta_1}{2} = \eta_1, 
\end{align*}
for all $s \in [1-\eta_2,1]$. Since $\eta_4 \leq \eta_1/2$, we also have 
$$
|\tilde u'(s)-\tilde u(1^-)|\leq \eta_1, \quad
\text{for a.e.\@ $s \in (1-\eta_2,1)$}.
$$
Combining the two previous inequalities and \eqref{3.65}, we obtain that for a.e. $s\in (1-\eta_2,1)$
$$
\frac{d}{ds}g(\tilde x'(s)) = \tau' \nabla g(\tilde x'(s))\cdot f(\tilde x'(s),\tilde u'(s)) \geq \frac{\zeta}{2}.
$$ 
Because $g(\tilde x'(1))=0$, we can conclude that $g(\tilde x'(s))<0$ for all $s\in [1-\eta_2,1)$. 
At this step, we have thus proved \eqref{tmp1-lemme-laurent} with $\eps_1:=\eta_4$. 

By similar arguments as above, there is $\eps_2>0$ such that for any $(\tilde u',\tau')\in  L^\infty(0,2;\R^m) \times (0,T)$ one has
\be{\label{tmp1-lemme-laurent2}}
\left\{
\begin{array}{ll}
\|\tilde u - \tilde u'\|_{L^\infty(0,2;\R^m)} &\leq \eps_2,\\
|\bar{\tau} - \tau'|&\leq \eps_2,\\
g(\tilde x'(1))&=0,
\end{array}
\right.
\quad \Rightarrow \quad 
\forall  s \in (1,2], \; g(\tilde x'(s))>0.
\ee
To conclude the proof, set $\eps:=\min(\eps_1,\eps_2)$ and take a pair $(\tilde u',\tau')\in  L^\infty(0,2;\R^m) \times (0,T)$ satisfying  
the inequalities $\|\tilde u - \tilde u'\|_{L^\infty(0,2;\R^m)} \leq \eps$ and $|\bar{\tau}-\tau'|\leq \eps$. 
It follows that $x'$ (the unique solution of \eqref{cauchy-pbm1} associated with the control $\tilde u' \circ \pi_{\tau'}^{-1}$) 
has exactly one crossing time over $[0,T]$ at $t=\tau'$. 
Because $\bar{u}$ is an optimal solution, we have $J_T(\bar{u}) \leq J_T(u')$, which can then be written 
(using the changes of variable $\pi_\tau$ and $\pi_{\tau'}$) as
$$
\phi(\tilde x(2))+T-\bar{\tau} = J_T(\bar{u}) \leq J_T(u')= \phi(\tilde x'(2))+T-\tau', 
$$
using that $\tilde x(2) = \bar{x}(T)$, $\tilde x'(2)=x'(T)$. 
This proves that $(\tilde u,\bar{\tau})$ is a weak minimum of \eqref{TC3}. 
\end{proof}

Note that an intermediate constraint is involved in problem \eqref{TC3} and that the data functions of \eqref{TC3} are all smooth.
At this stage, it is possible to derive optimality conditions for \eqref{TC3} (see, {\it{e.g.}}, \cite{dmitruk2}).

\subsection{Augmentation of the dynamics}{\label{augment-transfo-sec}}

The goal now is to formulate \eqref{TC3} over the fixed interval $[0,1]$ to avoid the use of the intermediate condition $g(\tilde x_{\tilde u,\tau}(1))=0$ (which will be replaced by an initial-final time condition), and so that we can use classical results of optimal control theory.
Hereafter, we use the notation 
$$
y:=
\left[
\begin{array}{c}
y^{(1)}\\
y^{(2)}\\
\xi
\end{array}
\right], \quad 
v:=
\left[
\begin{array}{c}
v^{(1)}\\
v^{(2)}\\
\end{array}
\right].
$$
for vectors in $\R^{2n+1}$ and in $\R^{2m}$ respectively. Consider the mappings 
$F:\R^{2n+1} \times \R^{2m} \rightarrow \R^{2n+1}$ (standing for an augmented dynamics) and 
$G:\R^{2n+1} \times \R^{2n+1} \rightarrow \R^{2n+2}$ (standing for a mixed initial-final constraint) defined respectively as
$$
F(y,v):=
\left[
\begin{array}{c}
\xi f(y^{(1)},v^{(1)})\\
(T-\xi) f(y^{(2)},v^{(2)})\\
0
\end{array}
\right] 
\quad \mathrm{and} \quad
G(y_0,y_1):=
\left[
\begin{array}{c}
y^{(1)}_0\\
\xi_0\\
y^{(2)}_0-y^{(1)}_1\\
g(y^{(1)}_1)
\end{array}
\right], 
$$
where $y_0:=(y^{(1)}_0,y^{(2)}_0,\xi_0)$ and $y_1:=(y^{(1)}_1,y^{(2)}_1,\xi_1)$. In this setting, 
the set of admissible controls is 
$$
\mathcal{V}:=\left\{v:=(v^{(1)},v^{(2)}):[0,1] \rightarrow U\times U \; ; \; v \; \mathrm{meas}.\right\}, 
$$
and we also define the set $C:=\{x_0\} \times (0,T) \times \{0_{\R^n} \}\times \{0\} \subset \R^{2n+2}$. 

\begin{rem}
The set $C$ comprises the initial condition at time $0$, the fact that $\tau \in (0,T)$ is free, the 
continuity of the trajectory at time $\tau$, and finally, the fact that the trajectory lies on the boundary of $K$ at time $\tau$
\end{rem}

The controlled dynamics then becomes
\be{\label{cauchy-pbm4}}
\frac{dy}{ds}(s)=F(y(s),v(s)),
\ee
with $v\in \mathcal{V}$ and $s\in [0,1]$. We denote by $\mathcal{T}$ the set of pairs $(y,v)$ satisfying \eqref{cauchy-pbm4}, with $v \in \mathcal{V}$.
Finally, we define a terminal pay-off $\psi:\R^{2n+1} \rightarrow \R$ of class $C^2$ as
$$
\psi(y)=\phi(y^{(2)})+T-\xi. 
$$
The new optimal control problem reads as follows:
\be{\label{TC4}}
\inf_{(y,v) \in \mathcal{T}}  \psi(y(1)) \quad \mathrm{s.t. } \;
G(y(0),y(1)) \in C.
\tag{$\mathrm{P}$}
\ee
Note that we keep the variable $y$ as an optimization variable, since its initial condition is not prescribed anymore and thus $y$ cannot be expressed as a function of the control $v$. Let us now recall the definition of a weak minimum and a Pontryagin minimum for \eqref{TC4}. 
 
\begin{defi}{\label{def2-min-faible}}
A pair $(\bar{y},\bar{v}) \in \mathcal{T}$ is a weak minimum (resp.\@ a Pontryagin minimum) of \eqref{TC4} if $G(\bar{y}(0),\bar{y}(1)) \in C$ and if there exists $\eps > 0$ such that for all $(y,v) \in \mathcal{T}$ satisfying
$G(y(0),y(1)) \in C$, one has:
\begin{equation} \label{eq:def_min_p}
|y(0)-\bar{y}(0)| \leq \eps  \quad \mathrm{and} \quad \| v - \bar{v} \|_{L^r(0,1;\R^{2m})} \leq \eps \; \; \Rightarrow \; \psi(\bar{y}(1)) \leq \psi(y(1)),
\end{equation}
for $r=\infty$ (resp.\@ $r= 1$).
\end{defi}

Problem \eqref{TC4} is a problem with a classical structure. The last step of the ``transformation" of the time crisis problem is done in the following proposition, where we construct a local solution to \eqref{TC4}.

\begin{prop}{\label{equivTC1-TC4}}
The pair $(\bar{y},\bar{v}) \in \mathcal{V}$, defined as follows, is a weak minimum of \eqref{TC4}:
\be{\label{def-control-v}}
\bar{y}(s)=\left|
\begin{array}{l}
\bar{y}^{(1)}(s):=\tilde x(s),\\
\bar{y}^{(2)}(s):=\tilde x(s+1), \\
\bar{\xi}(s):= \bar{\tau},
\end{array}
\right.
\quad
\bar{v}(s)=\left|
\begin{array}{l}
\bar{v}^{(1)}(s):=\tilde u(s),\\
\bar{v}^{(2)}(s):=\tilde u(s+1),
\end{array}
\right.
\quad s \in [0,1],
\ee
where $\bar{\tau}$ is the unique crossing time of $\bar{x}$ and where $\tilde u= \bar{u} \circ \pi_{\bar{\tau}}$ and $\tilde x= \bar{x} \circ \pi_{\bar{\tau}}$.
\end{prop}

\begin{proof} The pair $(\tilde u,\bar{\tau})$ is a weak minimum of \eqref{TC3}. Let then $\eps>0$ be as in Definition \ref{min-faible1}.
Since $\frac{d\pi_{\bar{\tau}}}{ds}= \bar{\tau}$ (resp.\@ $\frac{d\pi_{\bar{\tau}}}{ds} = T - \bar{\tau}$) over $[0,1]$ (resp.\@ over $[1,2]$), the pair $(\bar{y},\bar{v})$ satisfies \eqref{cauchy-pbm4}. In addition, it is easily seen that $\bar{y}$ satisfies
$G(\bar{y}(0),\bar{y}(1)) \in C$ since $\tilde x(0)=x_0$, $\bar{\tau} \in (0,T)$, $\bar{y}^{(1)}(1)=\bar{y}^{(2)}(0)=\tilde x(1)$, and $g(\tilde x(1))=0$. 
Let us now check that $(\bar{y},\bar{v})$ is is a weak minimum of \eqref{TC4} in the sense of Definition \ref{def2-min-faible}. Doing so, take a pair $(y,v) \in \mathcal{T}$ satisfying 
$$
\|v- \bar{v} \|_{L^\infty(0,1;\R^{2m})}\leq \eps, \quad
|y(0)- \bar{y}(0)| \leq \varepsilon,
$$ 
and such that $y(\cdot):=(y^{(1)}(\cdot),y^{(2)}(\cdot),\xi(\cdot))$ satisfies $G(y(0),y(1)) \in C$. Note that $\xi(\cdot)$ is constant over $[0,1]$ with $\xi(0) \in (0,T)$ as $G(y(0),y(1))\in C$. 
Consider the inverse transformation to the one used in \eqref{def-control-v} and define a pair   
$(\tilde x'(\cdot),\tilde u'(\cdot))$ over $[0,2]$ with $\tilde u' \in \mathcal{\tilde U}$, as well as a real number $\tau' \in (0,T)$ by
$$
\left|
\begin{array}{l}
\tilde x'(s):=y^{(1)}(s),\\
\tilde x'(s+1):=y^{(2)}(s),\\
\tau':=\xi(0),
\end{array}
\right.
\quad \mathrm{and} \quad
\left|
\begin{array}{l}
\tilde u'(s):=v^{(1)}(s),\\
\tilde u'(s+1):=v^{(2)}(s), 
\end{array}
\right.
\quad 
$$
for a.e.\@ $s \in [0,1]$. Using that $G(y(0),y(1))\in C$, 
we can check that $\tilde x'$ is continuous at $s=1$, that it is a solution of \eqref{cauchy-pbm3} associated with the control $\tilde u'$ and $\tau'$, that $\tilde x'(0)=x_0$, and that $g(\tilde x'(1))=0$. 
It follows that $\tilde u'$ is an admissible control for \eqref{TC3}.
In view of the relations between the controls in $\mathcal{V}$ and in $\tilde{\mathcal{U}}$, it is straightforward to check that the above transformation 
satisfies:
$$
\|\tilde u-\tilde u'\|_{L^\infty(0,2;\R^{m})} \leq \|v- \bar{v} \|_{L^\infty(0,1;\R^{2m})}\leq \eps \quad \text{and} \quad
|\tau-\tau'| = |\xi(0) - \bar{\xi}(0)| \leq |y(0)-\bar{y}(0)| \leq \eps.
$$ 
To conclude, since $\tilde u$ is a weak minimum of \eqref{TC3}, we deduce that 
$$
\phi(\tilde x(2))+T-\tau \leq \phi(\tilde x'(2))+T-\tau', 
$$
which is exactly saying that 
$\psi(\bar{y}(1))\leq \psi(y(1))$,
and that $(\bar{y},\bar{v})$ is a weak minimum of \eqref{TC4} as was to be proved. 
\end{proof}

\begin{rem}
The reformulation that we have performed remains valid for variants of the time crisis problem (in the sense that it would still yield an optimal control problem with a classical structure and smooth data functions).
For example, we could consider the case where 
admissible controls are with values in a subset $U_1$ when the state is in $K$, and in another subset $U_2$ when the state belongs to $K^c$. 
From a practical point of view, this situation typically happens if one is unable to use the same controls in both sets $K$ and $K^c$ 
(similar situations occur in sampled-data control, see, {\it{e.g.}}, {\rm{\cite{bourdin}}}).
It would also be possible to consider the situation with two different dynamics and two different integral costs on $K$ and $K^c$.
\end{rem}

\section{Necessary optimality conditions: case of a single crossing point}{\label{NOC-sec}}

We derive in this section first- and second-order optimality conditions for the pair $(\bar{y},\bar{v})$, defined by \eqref{def-control-v} and weak solution to \eqref{TC4}.
The obtained optimality conditions are then transformed into optimality conditions for the solution $\bar{u}$ to Problem \eqref{TC1}, with the help of the transformation that has been analyzed in Proposition \ref{equivTC1-TC4}.

As in the previous section, we work under Assumption (H3) with $r=1$. Since the third component of $\bar{y}$ is constant, we always denote it by $\bar{\xi}$ (instead of $\bar{\xi}(s)$).

\subsection{First-order optimality conditions} \label{NOC1-sec}

For the derivation of first-order optimality conditions for \eqref{TC4}, we introduce the following variables
\begin{equation*}
q:=
\left[
\begin{array}{c}
p^{(1)}\\
p^{(2)}\\
\lambda
\end{array}
\right]\in \R^{n+n+1}, \quad
\beta:=
\left[
\begin{array}{c}
\beta_1 \\ \beta_2 \\ \beta_3 \\ \beta_4
\end{array}
\right] \in \R^{n+1+n+1}, \quad
\mu:= \left[
\begin{array}{c}
\mu^{(1)} \\
\mu^{(2)}
\end{array}
\right] \in \R^{l+l}.
\end{equation*}
The variable $q$ denotes the co-state associated with $y$, the variable $\beta$ the Lagrange multiplier associated with $G$, and the variable $\mu$ denotes the Lagrange multiplier associated with the control constraints
\begin{equation*}
c(v^{(1)}(s)) \leq 0, \quad
c(v^{(2)}(s)) \leq 0, \quad \mathrm{for} \; \mathrm{a.e.} \; s\in [0,1]. 
\end{equation*}

The Hamiltonian associated with \eqref{TC4} is the function 
\be{\label{Hamil-TC4}}
\begin{array}{rl}
\hat H:\R^{2n+1} \times \R^{2n+1}\times \R^{2m} &\rightarrow \R \\
(y,q,v) &\mapsto q \cdot F(y,v). 
\end{array}
\ee
We also define the augmented Hamiltonian $\hat{H}^a$
\be{\label{Hamil_aug}}
\begin{array}{rl}
\hat{H}^a:\R^{2n+1} \times \R^{2n+1} \times \R^{2m} \times \R^{2l} &\rightarrow \R \\
(y,q,v,\mu) & \mapsto \hat{H}(y,q,v) + \mu^{(1)} \cdot c(v^{(1)}) + \mu^{(2)} \cdot c(v^{(2)}). 
\end{array}
\ee

In the following definition, $\R_+$ denotes the set of nonnegative real numbers.

\begin{defi} \label{defi:lag_mult_p}
A triplet $(\alpha,\beta,\mu) \in \R_+ \times \R^{2n+2} \times L^\infty(0,1;\R^{2l})$ is called \emph{Lagrange multiplier} (associated with $(\bar{y},\bar{v})$ and problem \eqref{TC4}) if the following conditions are satisfied:
\begin{itemize}
\item The triplet $(\alpha,\beta,\mu)$ is non-null, {\it{i.e.}}, $\alpha + |\beta| + \| \mu \|_{L^\infty(0,1;\R^{2l})} >0$ and is such that
\begin{equation} \label{normality-TC4}
\beta_2= 0, \quad
\mu(s) \geq 0, \quad
\mu^{(1)}(s) \cdot c(\bar{v}^{(1)}(s)) + \mu^{(2)}(s) \cdot c(\bar{v}^{(2)}(s))= 0
\quad \mathrm{for\ a.e.} \; s\in [0,1].
\end{equation}
\item There exists an absolutely continuous function $q:[0,1] \rightarrow \R^{2n+1}$ satisfying the following adjoint equation:
\be{\label{adjoint-TC4}}
\frac{dq}{ds}(s)=-\nabla_y \hat H(\bar{y}(s),q(s),\bar{v}(s)) \quad \mathrm{for\ a.e.} \; s\in [0,1],
\ee
and the following transversality conditions at $s=0$ and $s=1$:
\be{\label{transversality-TC4-0}}
\begin{array}{rl}
-q(0)&=\nabla_{y_0} G(\bar{y}(0),\bar{y}(1)) \beta, \vspace{0.1cm}\\
q(1)&=\nabla_{y_1} G(\bar{y}(0),\bar{y}(1)) \beta + \alpha \nabla \psi(\bar{y}(1)).
\end{array}
\ee
\item The augmented Hamiltonian is stationary with respect to $v$:
\be{\label{PMP-TC4}}
\nabla_v \hat{H}^a(\bar{y}(s),q(s),\bar{v}(s),\mu(s)) = 0, \\
\quad \mathrm{for\ a.e.} \; s\in [0,1].
\ee
\end{itemize}
\end{defi}

We denote by $\hat{\Lambda}_L(\bar{y},\bar{v})$ the set of Lagrange multipliers. Note that this set possibly contains singular Lagrange multipliers ({\it{i.e.}}, multipliers for which $\alpha= 0$). Let us mention that $\beta_2=0$, since the constraint $G_2$ is inactive at $(\bar{y}(0),\bar{y}(1))$. We also note that for all $(\alpha,\beta,\mu) \in \hat{\Lambda}_L(\bar{y},\bar{v})$ and for all $\theta > 0$, the triplet $(\theta \alpha,\theta \beta, \theta \mu)$ also lies in $\hat{\Lambda}_L(\bar{y},\bar{v})$. This  will enable us later to normalize Lagrange multipliers.

\begin{lem} \label{lem:oc_P}
The set of Lagrange multipliers $\hat{\Lambda}_L(\bar{y},\bar{v})$ is non-empty.
\end{lem}

\begin{proof}
Our proof is based on results of \cite{BDP14}.
In that reference, two kinds of multipliers are considered (see \cite[Definition 2.7]{BDP14}): Lagrange multipliers, corresponding to Definition \ref{defi:lag_mult_p} above, and Pontryagin multipliers, which are the Lagrange multipliers satisfying Pontryagin's Principle.

It is shown (see \cite[Theorem 3.1]{BDP14}) that for a Pontryagin minimum, the set of Pontryagin multipliers is non-empty and thus the set of Lagrange multipliers is also non-empty. The pair $(\bar{y},\bar{v})$ is a Pontryagin minimum to the following optimal control problem, obtained by adding to problem \eqref{TC4} a localizing constraint:
\begin{equation} \label{TC4_loc}
\inf_{(y,v) \in \mathcal{T}}  \psi(y(1)) \quad \mathrm{s.t. } \;
\begin{cases}
\begin{array}{l}
G(y(0),y(1)) \in C, \\
\| v - \bar{v} \|_{L^\infty(0,1;\R^{2l})} \leq \varepsilon,
\end{array}
\end{cases}
\end{equation}
where $\varepsilon > 0$ is such that \eqref{eq:def_min_p} holds true (with $r=\infty$). Since the localizing constraint is not active at $(\bar{y},\bar{v})$, the set $\hat{\Lambda}_L(\bar{y},\bar{v})$ is equal to the set of Lagrange multipliers associated with \eqref{TC4_loc}, which is non-empty by \cite[Theorem 3.1]{BDP14}. 

The application of \cite[Theorem 3.1]{BDP14} requires some regularity assumptions on the data, which are satisfied here by $f$, $g$, $c$, $\psi$, and the localizing constraint, and requires the inward pointing condition, which here directly follows from \eqref{IPC}. This concludes the proof.
\end{proof}

Next we rearrange the obtained optimality conditions, so that they can be easily translated into optimality conditions for the time crisis problem.
Note that the Hamiltonian can be expressed as follows: 
$$\hat H(y,q,v)
=\xi H(y^{(1)},p^{(1)},v^{(1)})+(T-\xi)H(y^{(2)},p^{(2)},v^{(2)}) + \lambda \cdot \xi,$$ 
where the Hamiltonian $H$ (associated with \eqref{TC1}) is defined as
\be{\label{Hamil-TC2}}
\begin{array}{rl}
H:\R^{n} \times \R^{n}\times \R^{m} &\rightarrow \R \\
(x,p,u) &\mapsto H(x,p,u):=p\cdot f(x,u). 
\end{array}
\ee
Note that this definition of the Hamiltonian does not take into account the indicator function.
Similarly, we can express the augmented Hamiltonian as follows:
\begin{equation*}
\hat{H}^a(y,q,v,\mu)
= \xi H^a \big( y^{(1)},p^{(1)},v^{(1)},{\textstyle \frac{1}{\xi} }\mu^{(1)} \big)
+ (T-\xi) H^a \big( y^{(2)},p^{(1)},v^{(2)}, {\textstyle \frac{1}{T-\xi} } \mu^{(2)} \big) + \lambda \cdot \xi,
\end{equation*}
where the augmented Hamiltonian $H^a$ (associated with \eqref{TC1}) is defined as
\be{\label{Hamil_aug-TC2}}
\begin{array}{rl}
H^a:\R^{n} \times \R^{n}\times \R^{m} \times \R^l & \rightarrow \R \\
(x,p,u,\nu) & \mapsto H^a(x,p,u,\nu):= p\cdot f(x,u) + \nu \cdot c(u).
\end{array}
\ee
Hence, conditions \eqref{adjoint-TC4}-\eqref{transversality-TC4-0}-\eqref{PMP-TC4} can be re-written as follows:
\begin{itemize}
\item The adjoint vector satisfies the following equations almost everywhere over $[0,1]$: 
$$
\begin{array}{cl}
\ds\frac{dp^{(1)}}{ds}&=-\xi \nabla_x H(\bar{y}^{(1)}(s),p^{(1)}(s),\bar{v}^{(1)}(s)), \vspace{0.1cm}\\
\ds\frac{dp^{(2)}}{ds}&=-(T-\xi) \nabla_x H(\bar{y}^{(2)},p^{(2)}(s),\bar{v}^{(2)}(s)),\vspace{0.1cm}\\
\ds\frac{d\lambda}{ds}&= -H(\bar{y}^{(1)}(s),p^{(1)}(s),\bar{v}^{(1)}(s)) + H(\bar{y}^{(2)}(s),p^{(2)}(s),\bar{v}^{(2)}(s)). 
\end{array}
$$
The Jacobian matrix of $G$ at $(\bar{y}(0), \bar{y}(1))$ is the following block matrix (the numbers indicated at the braces indicate the dimension of the row or the column):
$$
DG(\bar{y}(0),\bar{y}(1))={\small{
\overbrace{
\left[
\begin{array}{cccccc}
I_n & 0 & 0 & 0 &0 &0 \\
0 & 0& 1 & 0 & 0 &0 \\
0 & I_n &0 & -I_n & 0 & 0\\
0 & 0 & 0 & D g(\bar{y}^{(1)}(1)) & 0 & 0
\end{array}
\right]}^{(n,n,1,n,n,1)}
\left.
\vphantom{
\left[
\begin{array}{cccccc}
I_n & 0 & 0 & 0 &0 &0 \\
0 & 0& 1 & 0 & 0 &0 \\
0 & I_n &0 & -I_n & 0 & 0\\
0 & 0 & 0 & \nabla g(\bar{y}^{(1)}(1))^{\top} & 0 & 0
\end{array}
\right]}\right\}
\begin{array}{c}
n\\
1\\
n\\
1
\end{array}
}}
$$
thus, the transversality conditions read
$$
\begin{array}{rl}
-p^{(1)}(0)& = \beta_1,\\
-p^{(2)}(0)& = \beta_3,\\
-\lambda(0) & = \beta_2=0,
\end{array}
\quad \mathrm{and} \quad 
\begin{array}{rl}
p^{(1)}(1)&=-\beta_3+\beta_4 \nabla g(\bar{y}^{(1)}(1)),\\
p^{(2)}(1)&=\alpha \nabla \phi(\bar{y}^{(2)}(1)),\\
\lambda(1)&=-\alpha.
\end{array}
$$
\item The stationarity condition \eqref{PMP-TC4} is equivalent to
\begin{align*}
\nabla_{u} H^a(\bar{y}^{(1)}(s),p^{(1)}(s),\bar{v}^{(1)}(s),{\textstyle \frac{1}{\xi} } \mu^{(1)}(s))= 0,
\quad \mathrm{a.e.} \; s \in [0,1], \\
\nabla_{u} H^a(\bar{y}^{(2)}(s),p^{(2)}(s),\bar{v}^{(2)}(s),{\textstyle \frac{1}{T-\xi} } \mu^{(2)}(s))= 0,
\quad \mathrm{a.e.} \; s \in [0,1].
\end{align*}
\end{itemize}
Note that for a given $(\alpha,\beta,\mu) \in \hat \Lambda_L(\bar{y},\bar{v})$, there exists a unique costate $q$ satisfying the adjoint equation and the transversality conditions, since the latter defines in a unique way the terminal value $(p^{(1)}(1),p^{(2)}(1),\lambda(1))$.
We also deduce from the conditions $\lambda(0)=0$, $\lambda(1)= - \alpha$, and from the state equation that
\begin{equation} \label{noc_tau_TC4}
- \int_0^1 H(\bar{y}^{(1)}(s),p^{(1)}(s),\bar{v}^{(1)}(s))  \ \dd s + \int_0^1 H(\bar{y}^{(2)}(s),p^{(2)}(s),\bar{v}^{(2)}(s))  \ \dd s = - \alpha.
\end{equation}

We are now ready to convert the obtained optimality conditions  for the time crisis problem. In the definition below and in the second-order analysis, we make use of the following function:
\begin{equation} \label{eq:rho}
\rho_\tau(t)=
\frac{1}{\tau} \quad \text{ if $t \in (0,\tau)$}, \quad
\rho_\tau(t)= \frac{-1}{T-\tau} \quad \text{ if $t \in (\tau,T)$}.
\end{equation}

\begin{defi} \label{defi:lag_mult_tc}
A triplet $(\alpha,\gamma,\nu) \in \R_+ \times \R \times L^\infty(0,T;\R^{l})$ is called \emph{Lagrange multiplier} (associated with $\bar{u}$ and problem \eqref{TC1}) if the following conditions are satisfied:
\begin{itemize} 
\item The triplet $(\alpha,\gamma,\nu)$ is non-zero and the Lagrange multiplier $\nu$ satisfies the following sign and complementarity conditions:
\be{\label{sign_comp_TC2}}
\nu(t) \geq 0, \quad
\nu(t) \cdot c(\bar{u}(t))= 0, \quad
\quad \mathrm{for\ a.e.} \; t\in [0,T]. 
\ee
\item There exists a function $p \colon [0,T] \rightarrow \R^n$, whose restrictions to $[0,\bar{\tau})$ and $(\bar{\tau},T]$ are absolutely continuous, which
satisfies the following adjoint equation
\be{\label{adjoint-TC2}}
\begin{array}{ll}
\dot{p}(t)&= -\nabla_x H(\bar{x}(t),p(t),\bar{u}(t)) \quad \mathrm{for\ a.e.} \; t\in [0,T],\\
p(T)&= \alpha \nabla \phi(\bar{x}(T)),
\end{array}
\ee
and the following jump condition at the crossing time $\bar{\tau}$:
\be{\label{formule-saut0}}
p(\bar{\tau}^+)-p(\bar{\tau}^-) = \gamma \nabla g(\bar{x}(\bar{\tau})).
\ee
\item The augmented Hamiltonian is stationary with respect to $v$:
\be{\label{PMP-TC2}}
\nabla_u H^a(\bar{x}(t),p(t),\bar{u}(t),\nu(t))=0
\quad \mathrm{for\ a.e.} \; t\in [0,T]. 
\ee
\item The following relation holds true
\be{\label{noc_tau_TC2}}
\int_0^T \rho_{\bar{\tau}}(t) H(\bar{x}(t),p(t),\bar{u}(t))  \ \dd t
= \alpha. 
\ee
\end{itemize}
\end{defi}

We denote by $\Lambda_L(\bar{u},\bar{\tau})$ the set of Lagrange multipliers associated with $\bar{u}$ and problem \eqref{TC1}.

\begin{lem}{\label{ordre1-1}}
The set of Lagrange multipliers $\Lambda_L(\bar{u},\bar{\tau})$ is non-empty.
\end{lem}

\begin{proof}
These conditions follow from those obtained in Lemma \ref{lem:oc_P}.
Let us recall that the pair $(\bar{u},\bar{x})$ is related to $(v,y)$ by:
$$
\left|
\begin{array}{clll}
\bar{x}(t) & = \bar{y}^{(1)}(\pi_{\bar{\tau}}^{-1}(t)) & \mathrm{if} & t\in [0,\bar{\tau}],\\
\bar{x}(t) & = \bar{y}^{(2)}(\pi_{\bar{\tau}}^{-1}(t)) & \mathrm{if} & t\in [\bar{\tau},T],
\end{array}
\right.
\quad \mathrm{and} \quad 
\left|
\begin{array}{clll}
\bar{u}(t)& = \bar{v}^{(1)}(\pi_{\bar{\tau}}^{-1}(t)) & \mathrm{if} & t\in (0,\bar{\tau}),\\
\bar{u}(t)& = \bar{v}^{(2)}(\pi_{\bar{\tau}}^{-1}(t)) & \mathrm{if} & t\in (\bar{\tau},T). 
\end{array}
\right.
$$
The variable $\bar{\xi}$ is by construction constant, equal to $\bar{\tau}$.
We also recall that the optimality conditions obtained in Lemma \ref{lem:oc_P} involve multipliers $(\alpha,\beta,\mu)$. For the announced result, we keep the same value of $\alpha$, take $\gamma=\beta_4$ and define $p$ and $\nu$ as follows:
$$
\left|
\begin{array}{rlll}
p(t)&:=p^{(1)}(\pi_{\bar{\tau}}^{-1}(t)) & \mathrm{if}& t\in (0,{\bar{\tau}}), \\
p(t)&:=p^{(2)}(\pi_{\bar{\tau}}^{-1}(t)) & \mathrm{if}& t\in ({\bar{\tau}},T),
\end{array}
\right.
\quad
\left|
\begin{array}{rlll}
\nu(t)&:= {\textstyle \frac{1}{{\bar{\tau}}} } \mu^{(1)}(\pi_{\bar{\tau}}^{-1}(t)) & \mathrm{if}& t\in [0,{\bar{\tau}}), \\
\nu(t)&:= {\textstyle \frac{1}{T-{\bar{\tau}}}} \mu^{(2)}(\pi_{\bar{\tau}}^{-1}(t)) & \mathrm{if}& t\in ({\bar{\tau}},T],
\end{array}
\right.
$$
We note first that
\begin{align*}
\frac{dp(t)}{dt}= \frac{1}{{\bar{\tau}}} \frac{dp^{(1)}}{ds} \Big|_{s= \pi_{\bar{\tau}}^{-1}(t)} \text{ if } t\in (0,\bar{\tau}),
\qquad
\frac{dp(t)}{dt}= \frac{1}{T-{\bar{\tau}}} \frac{dp^{(2)}}{ds} \Big|_{s= \pi_{\bar{\tau}}^{-1}(t)} \text{ if } t\in ({\bar{\tau}},T).
\end{align*}
The costate equation \eqref{adjoint-TC2} follows then from \eqref{adjoint-TC4}.
Moreover, we have $p({\bar{\tau}}^-)=-\beta_3+\beta_4 \nabla g(\bar{x}({\bar{\tau}}))$ and $p({\bar{\tau}}^+)=-\beta_3$, thus the jump condition at ${\bar{\tau}}$ holds true. The terminal condition at time $T$ follows directly from the transversality condition $p^{(2)}(1)=\alpha \nabla \phi(\bar{y}^{(2)}(1))$.
Relations \eqref{sign_comp_TC2}, \eqref{PMP-TC2}, and \eqref{noc_tau_TC2} directly follow from \eqref{normality-TC4}, \eqref{PMP-TC4}, and \eqref{noc_tau_TC4}. It remains to check that $(\alpha,\gamma,\nu)$ is non-zero. If that was the case, then we would have $\mu= 0$ and $p=0$, implying then that $p^{(1)}$ and $p^{(2)}$ are null and finally that $\beta_1$ and $\beta_3$ are also null, which is not possible since $(\alpha,\beta,\mu)$ is itself non-zero.
\end{proof}

Let us mention that for a given $(\alpha,\gamma,\nu)$, the associated costate $p$ is unique, as can be easily verified.
We also mention that the construction detailed in the proof of Lemma \ref{ordre1-1} defines a bijection between $\hat{\Lambda}_L(\bar{y},\bar{v})$ and $\Lambda_L(\bar{u},\bar{\tau})$.

\subsection{Second-order optimality conditions}

We first introduce the critical cone $\hat{C}(\bar{y},\bar{v})$ and a quadratic form $\hat{\Omega}$ that will be used to formulate the second-order necessary optimality conditions associated with $(\bar{y},\bar{v})$. We will make use of the following notation:
\begin{equation*}
\delta y=
\left[
\begin{array}{c}
\delta y^{(1)} \\
\delta y^{(2)} \\
\delta \xi
\end{array}
\right], \quad
\delta v= \left[
\begin{array}{c}
\delta v^{(1)} \\
\delta v^{(2)}
\end{array}
\right].
\end{equation*}
Consider first the following differential system, obtained by linearizing the state equation \eqref{cauchy-pbm4}:
\begin{equation} \label{eq:lin_dynamics}
\frac{d}{ds} \delta y(s)= DF[s](\delta y(s),\delta v(s)), \quad \text{for a.e.\@ $s \in (0,1)$},
\end{equation}
where $[s]$ is used as a shortening of $(\bar{y}(s),\bar{v}(s))$.
The critical cone $\hat{C}(\bar{y},\bar{v})$ is now defined as follows:
\begin{equation*}
\hat{C}(\bar{y},\bar{v}):=
\left\{
\begin{array}{l}
(\delta y, \delta v) \in H^1(0,1;\R^{2n+1}) \times L^2(0,1;\R^{2l}),\, \big|\, \text{\eqref{eq:lin_dynamics} is satisfied,} \\[0.5em]
\qquad D \psi(\bar{y}(1)) \delta y(1) \leq 0, \\[0.5em]
\qquad D G_i(\bar{y}(0),\bar{y}(1))(\delta y(0), \delta y(1)) = 0, \text{ for $i=1,3,4$}, \\[0.5em]
\qquad c_i(\bar{v}^{(1)}(s))= 0 \Longrightarrow Dc_i(\bar{v}^{(1)}(s))\delta v^{(1)}(s)= 0, \text{ $\forall i=1,...,l$, for a.e.\@ $s \in (0,1)$ }, \\[0.5em]
\qquad c_i(\bar{v}^{(2)}(s))= 0 \Longrightarrow Dc_i(\bar{v}^{(2)}(s))\delta v^{(2)}(s)= 0, \text{ $\forall i=1,...,l$, for a.e.\@ $s \in (0,1)$ }
\end{array}
\right\}.
\end{equation*}
In the above definition, the space $H^1(0,1;\R^{2n+1})$ is the Sobolev space of time functions with a weak derivative in $L^2(0,1;\R^{2n+1})$. The constraint $G_2$ being inactive at $(\bar{y}(0),\bar{y}(1))$, it is not taken into account in the definition of the critical cone.
\begin{rem}
The set $\hat{C}(\bar{y},\bar{v})$ is in general referred to as \emph{strict} critical cone and the critical cone is actually a larger set than $\hat{C}(\bar{y},\bar{v})$. Still we use the terminology critical cone for $\hat{C}(\bar{y},\bar{v})$, for simplicity, and we note that the two cones are equal under a strict complementarity condition, see {\rm{\cite[Remark 4.8]{BDP14}}}.
\end{rem}
We also consider a quadratic form $\hat{\Omega}$, defined for $[\alpha,\beta,\mu] \in \hat{\Lambda}_L(\bar{y},\bar{v})$ and for $(\delta y, \delta v) \in H^1(0,1;\R^{2n+1}) \times L^2(0,1;\R^{2l})$ by
\begin{align}
\hat{\Omega}[\alpha,\beta,\mu](\delta y,\delta v)
:= \ & \alpha D^2 \psi(\bar{y}(1))(\delta y(1))^2
+ \beta \cdot D^2 G(\bar{y}(0),\bar{y}(1))(\delta y(0),\delta y(1))^2 \notag \\
& \qquad + \int_0^1 D^2 \hat{H}^a[s]( \delta y(s),\delta v(s))^2  \ \dd s,
\end{align}
where $q$ is the unique costate associated with $[\alpha,\beta,\mu]$.
In the above expression, the notation $[s]$ is a shortening of $(\bar{y}(s),q(s),\bar{v}(s),\mu(s))$. The Hessian of $\hat{H}^a$ is only considered with respect to the variables $y$ and $v$.
We have the following result.

\begin{lem} \label{lemma:cn2_p}
For all $(\delta y, \delta v) \in \hat{C}(\bar{y},\bar{v})$, there exists $(\alpha,\beta,\mu) \in \hat{\Lambda}_L(\bar{y},\bar{v})$ such that
\begin{equation}
\hat{\Omega}[\alpha,\beta,\mu](\delta y,\delta v) \geq 0.
\end{equation} 
\end{lem}

\begin{proof}
The result is a consequence of \cite[Theorem 4.9]{BDP14}. This result is originally stated for Pontryagin minima, but can be adapted to weak minima with the localization technique that was already invoked in the proof of Lemma \ref{lem:oc_P}. Regarding the required assumptions for the application of \cite[Theorem 4.9]{BDP14}, we have that Assumption 2 of \cite{BDP14} follows directly from the inward pointing condition \eqref{IPC}, Assumption 3 of \cite{BDP14} follows from the regularity assumptions on $f$, $g$, $\phi$, $g$, and $c$, Assumption 4 of \cite{BDP14} is irrelevant since we do not have pure state constraints, and Assumption 5 of \cite{BDP14} follows from \eqref{qualif2}.
\end{proof}

We perform now some calculations in order to transform the obtained optimality conditions into optimality conditions for the time crisis problem.
We first note that a pair $(\delta y, \delta v)$ satisfies the linearized dynamics \eqref{eq:lin_dynamics} if and only if
\begin{equation}
\begin{array}{rl}
\frac{d}{ds} \delta y^{(1)}(s)=  & \xi Df[s](\delta y^{(1)}(s),\delta v^{(1)}(s)) + \delta \xi f[s] \\[0.5em]
\frac{d}{ds} \delta y^{(2)}(s)=  & (T-\xi) Df[s](\delta y^{(2)}(s),\delta v^{(2)}(s)) - \delta \xi f[s] \\[0.5em]
\frac{d}{ds} \xi(s)= & 0.
\end{array}
\end{equation}
In the above relations, the notation $[s]$ is used as a shortening for $(\bar{y}^{(1)}(s),\bar{v}^{(1)}(s))$ and $(\bar{y}^{(2)}(s),\bar{v}^{(2)}(s))$, respectively.
For $(\delta y,\delta v)$ satisfying \eqref{eq:lin_dynamics}, the third component of $\delta y$ is constant and we therefore denote it by 
$\delta \xi$.
Regarding the conditions involved in the definition of the critical cone, we note that
\begin{equation}
D\psi(\bar{y}(1)) \delta y(1)= D \phi(\bar{y}^{(2)}(1)) \delta y^{(2)}(1) - \delta \xi,
\end{equation}
and that
\begin{align}
DG_i(\bar{y}(0),\bar{y}(1))(\delta y(0),\delta y(1))= 0, \text{ for $i=1,3,4$}
\qquad \Longleftrightarrow \qquad
\left\{
\begin{array}{rl}
\delta y^{(1)}(0)= & 0, \\[0.5em]
\delta y^{(2)}(0) - \delta y^{(1)}(1)= & 0, \\[0.5em]
Dg(\bar{y}^{(1)}(1)) \delta y^{(1)}(1)= & 0.
\end{array}
\right.
\end{align}
The first two terms involved in the definition of $\hat{\Omega}$ are given by
\begin{equation}
D^2 \psi(\bar{y}(1))(\delta y(1))^2
= D^2 \phi(\bar{y}^{(2)}(1)) (\delta y^{(2)}(1))^2
\end{equation}
and
\begin{equation}
\beta \cdot D^2 G(\bar{y}(0),\bar{y}(1))(\delta y(0),\delta y(1))
= \beta_4 D^2 g(\bar{y}^{(1)}(1))(\delta y(1))^2.
\end{equation}
We finally have the following expression for the Hessian of the augmented Hamiltonian:
\begin{align}
D^2 \hat{H}^a[s](\delta y(s),\delta v(s))^2
=\ & \xi D^2 H^a(\bar{y}^{(1)}(s),p^{(1)}(s),\bar{v}^{(1)}(s), {\textstyle \frac{1}{\xi} } \mu^{(1)}(s))(\delta y^{(1)}(s),\delta v^{(1)}(s)) \notag \\
& \ + (T-\xi) D^2 H^a(\bar{y}^{(2)}(s),p^{(2)}(s),\bar{v}^{(2)}(s), {\textstyle \frac{1}{T-\xi} } \mu^{(2)}(s))(\delta y^{(2)}(s),\delta v^{(2)}(s)) \notag \\
& \ + 2 \delta \xi DH(\bar{y}^{(1)}(s),p^{(1)}(s),\bar{v}^{(1)}(s)) (\delta y^{(1)}(s), \delta v^{(1)}(s)) \notag \\
& \ - 2 \delta \xi DH(\bar{y}^{(2)}(s),p^{(2)}(s),\bar{v}^{(2)}(s)) (\delta y^{(2)}(s), \delta v^{(2)}(s)).
\end{align}

We are prepared for defining, in an appropriate way, the linearized dynamics, the critical cone $C(\bar{u},\bar{\tau})$, and the quadratic form $\Omega(\bar{u},\bar{\tau})$ associated with the time crisis problem. Given $(\delta u, \delta \tau) \in L^2(0,T;\R^m) \times \R$, we consider the following linearized system:
\begin{equation} \label{eq:lin_sys_x}
\frac{d}{dt} \delta x(t)= 
Df[t](\delta x(t),\delta u(t))
+ \rho_{\bar{\tau}}(t) \delta \tau f[t], \quad \text{for a.e.\@ $t \in (0,T)$}, \quad
\delta x(0)= 0.
\end{equation}
We use the notation $[t]$ as a shortening of $(\bar{x}(t),\bar{u}(t))$. We recall that $\rho_{\bar{\tau}}$ was defined in \eqref{eq:rho}.
The critical cone $C(\bar{u},\bar{\tau})$ is defined as follows:
\begin{equation}
C(\bar{u},\bar{\tau}):=
\left\{
\begin{array}{l}
(\delta u, \delta \tau) \in L^2(0,2;\R^{2l}) \times \R \, \big|\, \text{for the solution $\delta x$ to \eqref{eq:lin_sys_x}, } \\[0.5em]
\qquad D \phi(\bar{x}(T)) \delta x(T) - \delta \tau \leq 0, \\[0.5em]
\qquad D g(\bar{x}(\tau)) \delta x(\bar{\tau}) = 0, \\[0.5em]
\qquad c_i(\bar{u}(t))= 0 \Longrightarrow Dc_i(\bar{u}(t)) \delta u(t)= 0, \text{ $\forall i=1,...,l$, for a.e.\@ $t \in (0,T)$ }
\end{array}
\right\}.
\end{equation}
Given $(\alpha,\gamma,\nu) \in \Lambda_L(\bar{u},\bar{\tau})$ and $(\delta u,\delta \tau) \in L^2(0,T;\R^m) \times \R$, we define the quadratic form $\Omega[\alpha,\gamma,\nu](\delta u,\delta \tau)$ as follows:
\begin{align}
\Omega[\alpha,\gamma,\nu](\delta u,\delta \tau)
= \ & \alpha D^2 \phi(\bar{x}(T)) (\delta x(T))^2 + \gamma D^2 g(\bar{x}(\bar{\tau})) (\delta x(\bar{\tau}))^2 \notag \\
& \qquad + \int_0^T D^2 H^a[t](\delta x(t),\delta u(t))^2  \ \dd t + 2 \delta \tau \int_0^T \rho_{\bar{\tau}}(t) DH[t](\delta x(t),\delta u(t))  \ \dd t,
\end{align}
where $\delta x$ denotes the solution to \eqref{eq:lin_sys_x} and where $[t]$ is a shortening of $(\bar{x}(t),p(t),\bar{u}(t),\nu(t))$. Like before, the first- and second-order derivatives of the Hamiltonians must be considered with respect to $(x,u)$ only.

\begin{lem} \label{lemma:cn2_tc}
For all $(\delta u,\delta \tau) \in C(u,\tau)$, there exists $(\alpha,\gamma,\nu) \in \Lambda_L(\bar{u},\bar{\tau})$ such that
\begin{equation} \label{eq:cn2_tc}
\Omega[\alpha,\gamma,\nu](\delta u,\delta \tau) \geq 0.
\end{equation}
\end{lem}

\begin{proof}
The lemma is a consequence of Lemma \ref{lemma:cn2_p} and of the calculations performed above. A bijection between $\hat{C}(\bar{y},\bar{v})$ and $C(\bar{u},\bar{\tau})$ can be established with the same transformation as the one used in the proof of Proposition \ref{equivTC1-TC4}. We recall that there is a bijection between $\hat{\Lambda}_L(\bar{y},\bar{v})$ and $\Lambda_L(\bar{u},\bar{\tau})$. Finally, if $(\delta u, \delta \tau)$ and $(\alpha,\gamma,\nu)$ are the images of $(\delta y, \delta v)$ and $(\alpha,\beta,\mu)$ via the two mentioned bijections, we have
\begin{equation*}
\Omega[\alpha,\gamma,\nu](\delta u, \delta \tau)
= \hat{\Omega}[\alpha,\beta,\mu](\delta y, \delta v).
\end{equation*}
which proves the lemma.
\end{proof}

\subsection{Optimality conditions in Pontryagin form}

The goal of this subsection is to improve the obtained optimality conditions by restricting the set of involved Lagrange multipliers in Lemma \ref{ordre1-1} and Lemma \ref{lemma:cn2_tc}.

\begin{defi}
We call Pontryagin multiplier a triplet $(\alpha,\gamma,\nu) \in \Lambda_L(u,\tau)$ which is such that Pontryagin's Principle holds, that is, for the associated costate $p$,
\begin{equation}
H(\bar{x}(t),p(t),\bar{u}(t)) \leq H(\bar{x}(t),p(t),u), \quad \forall u \in U, \quad \text{\rm{for a.e.}\@ $t \in (0,T)$}.
\end{equation}
\end{defi} 

We denote by $\Lambda_P(\bar{u},\bar{\tau})$ the set of Pontryagin multipliers.
We have now the following first- and second-order necessary optimality conditions in Pontryagin form.

\begin{thm} \label{theo_noc_strong}
The set $\Lambda_{P}(\bar{u},\bar{\tau})$ is non-empty. Moreover, for all $(\delta u, \delta \tau) \in C(\bar{u},\bar{\tau})$, there exists $(\alpha,\gamma,\nu) \in \Lambda_{P}(\bar{u},\bar{\tau})$ such that
\begin{equation*}
\Omega[\alpha,\gamma,\nu](\delta u,\delta \tau) \geq 0.
\end{equation*}
\end{thm}

The proof of the result is postponed to the end of the subsection and relies on three technical results. In the following lemma, we prove that $(\tilde{u},\bar{\tau})$ is a local minimizer for \eqref{TC3} in a stronger way than what has been stated in Proposition \ref{minGminF}.

\begin{lem} \label{lemma:mixed_optim}
For all $\omega \in (0,1)$, there exists $\varepsilon >0$ such that for all $(\tilde{u}',\tau') \in \tilde{\mathcal{U}} \times (0,T)$, the following implication holds true:
\begin{equation} \label{eq:mixed_optim}
\begin{cases}
\begin{array}{l}
\| \tilde{u} - \tilde{u}' \|_{L^\infty(1-\omega,1+\omega;\R^m)} \leq \varepsilon \\
\| \tilde{u} - \tilde{u}' \|_{L^1(0,1-\omega;\R^m)} \leq \varepsilon \\
\| \tilde{u} - \tilde{u}' \|_{L^1(1+\omega,2;\R^m)} \leq \varepsilon \\
|\bar{\tau} - \tau'| \leq \varepsilon
\end{array}
\Longrightarrow
\phi(\tilde x(2))+T-\tau \leq 
\phi(\tilde x'(2))+T-\tau',
\end{cases}
\end{equation}
where $\tilde{x}'= \tilde{x}_{\tilde{u}',\tau'}$.
\end{lem}

In words, we allow now perturbations of $\tilde{u}$ in the $L^1$-norm, except on the small interval $(1-\omega,1+\omega)$, where only a perturbation in $L^\infty$-norm is allowed. Let us mention that it is not possible to prove the result with $\omega= 0$, since around $s=1$, it is essential to perform a small perturbation in $L^\infty$-norm in order to preserve the structure of the trajectory (that is, the uniqueness of the crossing point). If the result was true with $\omega=0$, then we could prove that $(\bar{y},\bar{v})$ is a Pontryagin minimum for problem \eqref{TC4} and Theorem \ref{theo_noc_strong} would follow by direction application of \cite[Theorem 4.9]{BDP14}. 

\begin{proof}[Proof of Lemma \ref{lemma:mixed_optim}]
The proof is essentially the same as the one of Proposition \ref{minGminF}. We only indicate the modifications that have to be done.
First, one has to observe that the mapping defined in \eqref{eq:control_to_state_map} is still continuous when $\tilde{\mathcal{U}}$ is equipped with the $L^1$-norm (as can be easily verified with Gronwall's Lemma). Then, one can modify the left-hand side in the implication \eqref{eq:def_eta_4} used for the construction of $\eta_4$ as follows:
\begin{equation*}
\left\{
\begin{array}{ll}
\|\tilde u - \tilde u'\|_{L^1(0,1-\omega;\R^m)} & \leq \eta_4\\
\|\tilde u - \tilde u'\|_{L^\infty(1-\omega,1;\R^m)} & \leq \eta_4\\
|\bar{\tau}-\tau'|&\leq \eta_4
\end{array}
\right.
\; \Rightarrow \quad 
\begin{array}{l}
\forall s\in [0,1-\eta_2], \; |g(\tilde x'(s))- g(\tilde x(s)) |< \eta_3, \\
\forall s\in [1-\eta_2,1], \; |\tilde x(s)-\tilde x'(s)|\leq \frac{\eta_1}{2}, 
\end{array}
\end{equation*}
The rest of the proof is identical.
\end{proof}

Given $\omega > 0$, we define the following set of multipliers:
\begin{align}
\Lambda_{P,\omega}(\bar{u},\bar{\tau})
= \big\{ (\alpha,\gamma,\nu) \in \Lambda_L(\bar{u},\bar{\tau}) \,;\, &
H(\bar{x}(t),p(t),\bar{u}(t)) \leq H(\bar{x}(t),p(t),u), \text{ for all $u \in U$}, \notag \\
& \text{for a.e. $t \in (0,\bar \tau(1-\omega)) \cup (\bar \tau + \omega(T-\bar \tau),T)$} \big\}.
\end{align}
In words, $\Lambda_{P,\omega}(\bar{u},\bar{\tau})$ is the set of Lagrange multipliers satisfying Pontryagin's Principle on the whole time interval $(0,T)$, except on $(\bar \tau(1 - \omega), \bar \tau + \omega(T-\bar \tau))$.
The approach that we propose now is the following. Lemma \ref{lemma:mixed_optim} will enable us to restrict the set of Lagrange multipliers involved in Lemma \ref{ordre1-1} and Lemma \ref{lemma:cn2_tc} to the set $\Lambda_{P,\omega}(\bar{u},\bar{\tau})$, for a given value of $\omega>0$. Theorem \ref{theo_noc_strong} will be obtained by ``passing to the limit" when $\omega \to 0$.

\begin{lem} \label{lemm:lambda_p_omega}
For all $\omega>0$, the set $\Lambda_{P,\omega}(\bar{u},\bar{\tau})$ is non-empty.
\end{lem}

\begin{proof}
Let $\omega > 0$. Let $\varepsilon>0$ be such that \eqref{eq:mixed_optim} holds true. Then, the pair $(\bar{y},\bar{v})$ is a Pontryagin minimum of the following localized problem:
\begin{equation} \label{TC4_loc2}
\inf_{(y,v) \in \mathcal{T}}  \psi(y(1)) \quad \mathrm{s.t. } \;
\begin{cases}
\begin{array}{l}
G(y(0),y(1)) \in C, \\
\| v^{(1)} - \bar{v}^{(1)} \|_{L^\infty(1-\omega,1;\R^{l})} \leq \varepsilon, \\
\| v^{(2)} - \bar{v}^{(2)} \|_{L^\infty(0,\omega;\R^{l})} \leq \varepsilon.
\end{array}
\end{cases}
\end{equation}
Applying \cite[Theorem 3.1]{BDP14}, we obtain the existence of $(\alpha,\beta,\mu) \in \Lambda_L(\bar{y},\bar{v})$ such that
\begin{align*}
H(\bar{y}^{(1)}(s),p^{(1)}(s),\bar{v}^{(1)}(s)) \leq H(\bar{y}^{(1)}(s),p^{(1)}(s),v), \quad & \forall v \in U, \quad \text{for a.e.\@ $s \in (0,1-\omega)$}, \\
H(\bar{y}^{(2)}(s),p^{(2)}(s),\bar{u}^{(2)}(s)) \leq H(\bar{y}^{(2)}(s),p^{(2)}(s),v), \quad & \forall v \in U, \quad \text{for a.e.\@ $s \in (\omega,1)$}.
\end{align*}
The non-emptiness of $\Lambda_{P,\omega}(\bar{u},\bar{\tau})$ follows then with the usual transformation.
\end{proof}

Let us equip $\Lambda_L(\bar{u},\bar{\tau})$ with the following norm:
\begin{equation*}
\| (\alpha,\gamma,\nu) \|:= \alpha + |\beta_4| + \| \nu \|_{L^1(0,T;\R^l)}.
\end{equation*}
The following lemma is a technical lemma which will be useful for the announced passage to the limit.

\begin{lem} \label{lemma:lim_pontry}
Let $(\omega_k)_{k \in \mathbb{N}}$ be a sequence of positive numbers converging to 0. Let $(\alpha^k,\gamma^k,\nu^k)_{k \in \mathbb{N}}$ be a sequence in $\Lambda_L(\bar{u},\bar{\tau})$ such that
\begin{equation*}
\| (\alpha^k,\gamma^k,\nu^k) \|=1 \quad \text{and} \quad
(\alpha^k,\gamma^k,\nu^k) \in \Lambda_{P,\omega_k}(\bar{u},\bar{\tau}), \quad \forall k \in \mathbb{N}.
\end{equation*}
Then, there exists at least one non-zero weak-$*$ limit point that belongs to $\Lambda_{P}(\bar{u},\bar{\tau})$.
\end{lem}

\begin{proof}
The proof is in line with \cite[Lemma 3.5]{BDP14}.
As a consequence of the inward pointing condition, we have that the sequence $\nu_k$ is also bounded for the $L^\infty$ norm (see \cite[Theorem 3.1]{BO10}). The existence of a weak-$*$ limit point $(\alpha,\gamma,\nu)$ follows. Without loss of generality, we now assume that the whole sequence converges to $(\alpha,\gamma,\nu)$ for the weak-$*$ topology.
Let us first check that $(\alpha,\gamma,\nu)$ is non-zero. We denote by $\mathbf{1}$ the vector of dimension $l$ with coordinates equal to 1. Since $\nu_k \geq 0$, we have $\| \nu_k \|_{L^1(0,T;\R^l)} = \int_0^T \mathbf{1} \cdot \nu_k(t) \, \text{d} t$. Thus we can pass to the limit in the $L^1$-norm, which guarantees that
\begin{equation*}
\| (\alpha,\gamma,\nu) \|= \lim_{k \to \infty} \| (\alpha^k,\gamma^k,\nu^k) \| = 1.
\end{equation*}
Let us prove now that $(\alpha,\gamma,\nu) \in \Lambda_L(\bar{u},\bar{\tau})$.
The reader can easily verify that $\alpha \geq 0$ and that \eqref{sign_comp_TC2} holds true.
We denote by $p^k$ the unique costate associated with $(\alpha^k,\gamma^k,\nu^k)$, for all $k \in \mathbb{N}$. We denote by $p$ the unique costate associated with $(\alpha,\gamma,\nu)$. Let us prove that $(p^k)_{k \in \mathbb{N}}$ converges uniformly to $p$. The costates $p^k$ are all solutions to the same differential equation on $(\bar{\tau},T]$, with terminal condition $\alpha_k \nabla \phi(\bar{x}(T))$. Since $\alpha_k \nabla \phi(\bar{x}(T)) \rightarrow \alpha \nabla \phi(\bar{x}(T))$, we deduce that $(p^k_{|(\bar{\tau},T]})$ converges uniformly to $p_{|(\bar{\tau},T]}$. Then, since $\gamma^k \rightarrow \gamma$, we obtain that $p^k(\bar{\tau}^-) \rightarrow p(\bar{\tau}^-)$ and with the same argument as before, we obtain that $(p^k)$ converges uniformly to $p$ on $[0,\bar{\tau})$. Now, we observe that the sequence $\nabla H^a(\bar{x}(\cdot),p^k(\cdot),\bar{u}(\cdot),\nu^k(\cdot))$ converges to $\nabla H^a(\bar{x}(\cdot),p(\cdot),\bar{u}(\cdot),\nu(\cdot))$ for the weak-$*$ star convergence, since $H^a$ is linear in $\nu$ and since $c(\bar{u}(\cdot))$ lies in $L^\infty(0,T;\R^l)$. Thus $\nabla H^a(\bar{x}(\cdot),p(\cdot),\bar{u}(\cdot),\nu(\cdot))$ is null. With the uniform convergence of $(p^k)_{k \in \mathbb{N}}$, we directly obtain that \eqref{noc_tau_TC2} holds true. It follows that $(\alpha,\gamma,\nu) \in \Lambda_L(\bar{u},\bar{\tau})$.

To conclude the proof, it remains to prove that Pontryagin's Principle is satisfied. Let $\tau_1 \in (0,\bar{\tau})$ and $\tau_2 \in (\bar{\tau},T)$ be arbitrary. It suffices to show that
\begin{equation} \label{ineq_pont_0}
\begin{array}{ll}
H(\bar{x}(t),p(t),u(t)) \leq H(\bar{x}(t),p(t),u), \quad & \forall u \in U, \quad \text{for a.e. $t \in (0,\tau_1)$}, \\
H(\bar{x}(t),p(t),u(t)) \leq H(\bar{x}(t),p(t),u), \quad & \forall u \in U, \quad \text{for a.e. $t \in (\tau_2,T)$}.
\end{array}
\end{equation}
Let $\bar{k}$ be sufficiently large, so that $(0,\tau_1) \subseteq (0,\bar \tau(1-\omega_k))$, for all $k \geq \bar{k}$. For all $k \geq \bar{k}$, we denote by $I_k$ the set of times $t \in (0,\tau_1)$ such that
\begin{equation} \label{ineq_pont}
H(\bar{x}(t),p^k(t),u(t)) \leq H(\bar{x}(t),p^k(t),u), \quad \forall u \in U.
\end{equation}
Since $(\alpha^k,\gamma^k,\nu^k) \in \Lambda_{P,\omega_k}(\bar{u},\bar{\tau})$, we have that $(0,\tau_1) \backslash I_k$ is of zero measure.
Let $I= \bigcap_{k \geq \bar{k}} I_k$. We have $(0,\tau_1) \backslash I = \bigcup_{k \geq \bar{k}} \big( (0,\tau_1) \backslash I_k \big)$, thus $(0,\tau_1) \backslash I$ is a set of zero measure.
Passing to the limit w.r.t.\@ $k$ in inequality \eqref{ineq_pont}, for all $t \in I$, we obtain the first inequality in \eqref{ineq_pont_0}. The second inequality is proved similarly.
This concludes the proof. 
\end{proof}

\begin{proof}[Proof of Theorem \ref{theo_noc_strong}]
Let $\omega > 0$. Considering the localized problem introduced in \eqref{TC4_loc2} in the proof of Lemma \ref{lemm:lambda_p_omega}, we obtain that the second-order optimality conditions derived in Lemma \ref{lemma:cn2_tc} are still valid if $\Lambda_L(\bar{u},\bar{\tau})$ is replaced by $\Lambda_{P,\omega}(\bar{u},\bar{\tau})$ in inequality \eqref{eq:cn2_tc}. Now, let $(\delta u,\delta \tau) \in C(\bar{u},\bar{\tau})$ and let $(\omega_k)_{k \in \mathbb{N}}$ be a sequence of positive numbers converging to 0. As we have just proved, 
for all $k \in \mathbb{N}$, there exists  a multiplier $(\alpha^k,\gamma^k,\nu^k) \in \Lambda_{P,\omega_k}(\bar{u},\bar{\tau})$ such that $\Omega[\alpha^k,\gamma^k,\nu^k](\delta u,\delta \tau) \geq 0$. Re-normalizing if necessary the sequence $(\alpha^k,\gamma^k,\nu^k)_{k \in \mathbb{N}}$, we can assume that $\| \alpha^k,\gamma^k,\nu^k \|=1$. Applying then Lemma \ref{lemma:lim_pontry}, we obtain the existence of a weak-$*$ limit point $(\alpha,\gamma,\nu) \in \Lambda_P(\bar{u},\bar{\tau})$. Finally, the quadratic form $\Omega$ is weakly-$*$ continuous with respect to $(\alpha,\gamma,\nu)$ because $H^a$ is linear w.r.t.~the multiplier $\nu$ and the costates $p^k$ associated with $(\alpha^k,\gamma^k,\nu^k)$ uniformly converge to the costate $p$ associated with $(\alpha,\gamma,\nu)$. 
Thus we have
\begin{equation*}
\Omega[\alpha,\gamma,\nu](\delta u,\delta \tau)
= \lim_{k \to \infty} \Omega[\alpha^k,\gamma^k,\nu^k](\delta u,\delta \tau)
\geq 0,
\end{equation*}
which concludes the proof of the theorem.
\end{proof}

\begin{rem}
It is possible to provide second-order sufficient optimality conditions for problem $(\mathrm{P})$ (see {\rm{\cite{BDP14b}}}). These conditions would consist in a natural strengthening of the necessary optimality conditions and would ensure that $(\bar{y},\bar{v})$ is a Pontryagin minimum. However, they would not guarantee the local optimality of $\bar{u}$ (for problem \eqref{TC1}) for any standard topology.
\end{rem}

\subsection{Non-singularity of the Pontryagin multipliers}

\begin{lem} \label{lemma:constancy}
Let $(\alpha,\gamma,\nu) \in \Lambda_P(\bar{u},\bar{\tau})$. Let $p$ be the associated costate.
Then, there exist $H_1$ and $H_2 \in \R$ such that
\begin{equation} \label{eq:constancy}
H(\bar{x}(t),p(t),\bar{u}(t))= H_1, \quad \text{for a.e.\@ $t \in (0,\bar{\tau})$}, \quad
H(\bar{x}(t),p(t),\bar{u}(t))= H_2, \quad \text{for a.e.\@ $t \in (\bar{\tau},T)$}.
\end{equation}
Moreover,
\be{\label{saut-Hamiltonien-TC4}}
H_2-H_1=-\alpha.
\ee
\end{lem}

\begin{proof}
The proof of \eqref{eq:constancy} follows from the classical proof of the constancy of the Hamiltonian along extremal trajectories (solutions of a Hamiltonian system issued from Pontryagin's Principle) for autonomous problems, see, {\it{e.g.}}, \cite{fatorini}. 
Equality \eqref{saut-Hamiltonien-TC4} follows directly from \eqref{noc_tau_TC2}. 
\end{proof}

We are now able to prove that Pontryagin multipliers are not singular and unique, up to a multiplicative constant.

\begin{prop} \label{prop-anormal}
For all $(\alpha,\gamma,\nu) \in \Lambda_P(\bar{u},\bar{\tau})$, we have $\alpha>0$. Moreover, there exists a unique Pontryagin multiplier such that $\alpha=1$.
\end{prop}

\begin{proof}
Let $(\alpha,\gamma,\nu) \in \Lambda_P(\bar{u},\bar{\tau})$ be such that $\alpha=0$. Let $p$ be the associated costate. Then, $p(T)=0$ and thus $p(t)=0$ for all $t \in (\bar{\tau},T]$. It follows from Lemma \ref{lemma:constancy} that $H_2=0$ and thus $H_1= H_2+ \alpha= 0$. We deduce then from the jump condition that
\begin{equation*}
\gamma \nabla g(\bar{x}(\bar{\tau})) \cdot f(\bar{x}(\bar{\tau}),\bar{u}(\bar{\tau}^-))
= (\underbrace{p(\bar{\tau}^+)}_{=0}-p(\bar{\tau}^-)) \cdot f(\bar{x}(\bar{\tau}),\bar{u}(\bar{\tau}^-))
= - H_1 = 0.
\end{equation*}
Since $\bar{\tau}$ is transverse, the scalar product $\nabla g(\bar{x}(\bar{\tau})) \cdot f(\bar{x}(\bar{\tau}),\bar{u}(\bar{\tau}^-))$ is non-zero and therefore $\gamma=0$. Using again the jump condition, we deduce then that $p(\bar{\tau}^-)= p(\bar{\tau}^+)=0$ and thus that $p=0$ on $[0,\bar{\tau})$. It further follows from the stationarity of the augmented Hamiltonian that
\begin{equation*}
\nabla c(\bar{u}(t)) \nu(t)=0.
\end{equation*}
Let us set $\Delta_{c,i}:= \{ t \in (0,T) \,| \, c_i(\bar{u}(t))= 0 \}$. By the complementarity condition, we have that
\begin{equation*}
\nu_i(t)= 0, \quad \forall i=1,...,\ell, \quad \text{for a.e.\@ $t \in (0,T) \backslash \Delta_{c,i}$}.
\end{equation*}
By the surjectivity condition \eqref{qualif2} (deduced from Assumption (H1)), we have that there exists $v \in L^2(0,T;\R^m)$ such that
\begin{equation*}
Dc_i(\bar{u}(t))v(t)= \nu_i(t), \quad \forall i=1,...,l, \quad \text{for a.e.\@ $t \in \Delta_{c,i}$}.
\end{equation*}
Therefore,
\begin{equation} \label{surj_app}
0=
\sum_{i=1}^l
\int_{\Delta_{c,i}}
\nabla c_i(\bar{u}(t))\nu_i(t) \cdot v(t)  \ \dd t
= \sum_{i=1}^l
\int_{\Delta_{c,i}}
\nu_i(t)  Dc_i(\bar{u}(t)) v(t)  \ \dd t
= \sum_{i=1}^l \int_{\Delta_{c,i}} | \nu_i (t) |^2  \ \dd t.
\end{equation}
We conclude that $\nu_i(t)=0$ for all $i=1,...,l$ and for a.e.\@ $t \in \Delta_{c,i}$ and thus that $\nu=0$. We obtain a contradiction with the non-nullity of $(\alpha,\gamma,\nu)$. We can conclude that $\alpha>0$. 

The existence of a Pontryagin multiplier with $\alpha=1$ follows. Consider now two Pontryagin multipliers $(\alpha,\gamma,\nu)$ and $(\alpha',\gamma',\nu')$ with $\alpha= \alpha'=1$. Denote by $p$ and $p'$ the associated costates. Denote by $H_1$, $H_2$, $H_1'$ and $H_2'$ the constant values of the corresponding Hamiltonians on $(0,\bar{\tau})$ and $(\bar{\tau},T)$. We first observe that $p$ and $p'$ are equal on $(\bar{\tau},T]$. Thus $H_2= H_2'$ and since $\alpha= \alpha'$, we also have that $H_1= H_1'$. Then, using the jump condition at $\bar{\tau}$ and the equality $p(\bar{\tau}^+)=p'(\bar{\tau}^+)$, we obtain that
\begin{align*}
\gamma \nabla g(\bar{x}(\bar{\tau})) \cdot f(\bar{x}(\bar{\tau}),\bar{u}(\bar{\tau}^-))
= \ & (p(\bar{\tau}^+)-p(\bar{\tau}^-)) \cdot f(\bar{x}(\bar{\tau}),\bar{u}(\bar{\tau}^-)) \\
= \ & p'(\bar{\tau}^+) \cdot f(\bar{x}(\bar{\tau}),\bar{u}(\bar{\tau}^-)) - H_1 \\
= \ & p'(\bar{\tau}^+) \cdot f(\bar{x}(\bar{\tau}),\bar{u}(\bar{\tau}^-)) - H_1' \\
= \ & (p'(\bar{\tau}^+)-p'(\bar{\tau}^-)) \cdot f(\bar{x}(\bar{\tau}),\bar{u}(\bar{\tau}^-)) \\
= \ & \gamma' \nabla g(\bar{x}(\bar{\tau})) \cdot f(\bar{x}(\bar{\tau}),\bar{u}(\bar{\tau}^-)).
\end{align*}
We conclude that $\gamma= \gamma'$, by the transversality of $\bar{\tau}$. It follows that $p=p'$ on $[0,\bar{\tau}^-)$ and thus,
\begin{equation*}
\nabla c(\bar{u}(t)) \nu(t)= \nabla c(\bar{u}(t)) \nu'(t), \quad \text{for a.e.\@ $t \in (0,T)$}.
\end{equation*}
Using that the mapping \eqref{qualif2} is surjective and proceeding as in \eqref{surj_app}, we obtain that $\nu= \nu'$, which concludes the proof of uniqueness.
\end{proof}

\begin{rem} \label{true_constancy}
The Hamiltonian $H$ does not contain the indicator function of the set $K^c$. Let us define $H_0(x,p,u)= H(x,p,u) + \mathds{1}_{K^c}(x)$. The mapping $H_0$ can be seen as the ``true" Hamiltonian associated with the time crisis problem. We deduce from Lemma \ref{lemma:constancy} that for the unique Pontryagin multiplier with $\alpha=1$, we have that $t\mapsto H_0(\bar{x}(t),p(t),\bar{u}(t))$ is constant almost everywhere over the whole interval $(0,T)$.
\end{rem}

\begin{rem}
The proof of Proposition \ref{prop-anormal} uses in an essential manner the property of constancy of the Hamiltonian obtained in Lemma \ref{lemma:constancy}. Therefore, Proposition \ref{prop-anormal} cannot be extended in a direct way to Lagrange multipliers.
\end{rem}

\section{Necessary optimality conditions: case of several crossing points} \label{NOC_sec_gen}

We extend in this section the obtained results to the situation with several crossing points, without detailing proofs. We therefore assume that (H3) is satisfied with crossing points $\bar{\tau}_1 < ... < \bar{\tau}_r$. We denote by $\bar{\tau} \in (0,T)^r$ the vector $(\bar{\tau}_1,...,\bar{\tau}_r)$ and make use of the conventions $\bar{\tau}_0=0$ and $\bar{\tau}_{r+1}= T$.

We first need to define a new change of variables. Given $\tau \in (0,T)^r$, we define the mapping $\pi_\tau \colon s \in [0,r+1] \rightarrow [0,T]$ as follows:
\be{\label{time-change1_gen}}
\pi_\tau(s):=
\tau_j + (s-j)(\tau_{j+1}-\tau_j), \quad \forall j=0,...,r, \quad \forall s \in [j,j+1].
\ee
Given a control $\tilde{u} \in L^\infty(0,r+1;U)$ and $\tau \in (0,T)^r$, there is a unique solution $\tilde x_{\tilde u,\tau}$ of the Cauchy problem
\be{\label{cauchy-pbm3_gen}}
\left\{
\begin{array}{rl}
\ds \frac{d\tilde x}{ds}(s)&=\ds \frac{d\pi_{\tau}}{ds}(s) f(\tilde x(s),\tilde u(s)) \quad \mathrm{for\ a.e.} \; s\in [0,r+1],\vspace{0.15cm}\\
\tilde x(0)&=x_0. 
\end{array}
\right.
\ee
The optimal control problem to be considered, after change of variable is now
\be{\label{TC3_gen}}
\inf_{ \tilde u \in  \mathcal{\tilde U}, \; \tau \in (0,T)^r} \phi(\tilde x_{\tilde u,\tau}(r+1))+ \sum_{j=1}^r (-1)^j \tau_j \quad \mathrm{s.t.} \; g(\tilde x_{\tilde u,\tau}(j))=0, \quad \forall j=1,...,r. 
\ee
For the generalization of the first- and second-order optimality conditions, we re-define the mapping $\rho_{{\tau}}$ as a mapping in $L^\infty(0,T;\R^{r})$ as follows:
\begin{equation*}
(\rho_{{\tau}}(t))_j
= \frac{1}{{{\tau}}_j- {{\tau}}_{j-1}} \quad \text{ if $t \in (j-1,j)$}, \quad
(\rho_{{\tau}}(t))_j
= \frac{-1}{{{\tau}}_{j+1}-{{\tau}}_j} \quad \text{ if $t \in (j,j+1)$}, \quad
(\rho_{{\tau}}(t))_j
= 0 \quad \text{otherwise}.
\end{equation*}

The following result extends the results obtained in Section \ref{NOC-sec}.

\begin{thm}
There exists a unique pair $(\gamma,\nu) \in \R^r \times L^\infty(0,T;\R^l)$ satisfying the following properties.
\begin{itemize}
\item The Lagrange multiplier $\nu$ satisfies \eqref{sign_comp_TC2}.
\item There exists a function $p \colon [0,T] \rightarrow \R^n$, whose restrictions to $[0,\bar{\tau}_1)$, $(\bar{\tau}_1,\bar{\tau}_2)$,...,$(\bar{\tau}_r,T]$ are absolutely continuous, which
satisfies the following adjoint equation
\be{\label{adjoint-TC2_gen}}
\begin{array}{ll}
\dot{p}(t)&= -\nabla_x H(\bar{x}(t),p(t),\bar{u}(t)) \quad \mathrm{a.e.} \; t\in [0,T],\\
p(T)&= \nabla \phi(\bar{x}(T)),
\end{array}
\ee
and the following jump conditions at the crossing times $\bar{\tau}_j$, $1 \leq j \leq r$:
\be{\label{formule-saut0_gen}}
p(\bar{\tau}_j^+)-p(\bar{\tau}_j^-) = \gamma_j \nabla g(\bar{x}(\bar{\tau}_j)).
\ee
\item The augmented Hamiltonian is stationary with respect to $v$, i.e.,\@ it satisfies \eqref{PMP-TC2}.
\item The following relation holds true
\be{\label{noc_tau_TC2_gen}}
\int_0^T (\rho_{\bar{\tau}}(t))_j H(\bar{x}(t),p(t),\bar{u}(t))  \ \dd t + (-1)^j= 0, \quad \forall j=1,...,r. 
\ee
\end{itemize}
Moreover, the mapping $t \in (0,T) \mapsto H_0(\bar{x}(t),p(t),\bar{u}(t))$ is constant (with $H_0$ defined as in Remark \ref{true_constancy}).
\end{thm}

The extension of the second order optimality conditions follows the same lines. The linearized dynamics, for $\delta u \in L^\infty(0,T;\R^m)$ and $\delta \tau \in \R^r$ reads:
\begin{equation} \label{eq:lin_sys_x_gen}
\frac{d}{dt} \delta x(t)= 
Df[t](\delta x(t),\delta u(t))
+ (\rho_{\bar{\tau}}(t) \cdot \delta \tau) f[t], \quad \text{for a.e.\@ $t \in (0,T)$}, \quad
\delta x(0)= 0.
\end{equation}
The critical cone $C(\bar{u},\bar{\tau})$ is defined as follows:
\begin{equation}
C(\bar{u},\bar{\tau}):=
\left\{
\begin{array}{l}
(\delta u, \delta \tau) \in L^2(0,2;\R^{2l}) \times \R^r \, \big|\, \text{for the solution $\delta x$ to \eqref{eq:lin_sys_x}, } \\[0.5em]
\qquad D \phi(\bar{x}(T)) \delta x(T) + \sum_{j=1}^r (-1)^j \delta \tau_j \leq 0, \\[0.5em]
\qquad D g(\bar{x}(\bar{\tau}_j)) \delta x(\bar{\tau}_j) = 0, \; j=1,...,r,\\[0.5em]
\qquad c_i(\bar{u}(t))= 0 \Longrightarrow Dc_i(\bar{u}(t)) \delta u(t)= 0, \text{ $\forall i=1,...,l$, for a.e.\@ $t \in (0,T)$ }
\end{array}
\right\}.
\end{equation}
The quadratic form $\Omega[1,\gamma,\nu](\delta u,\delta \tau)$ is defined as
\begin{align}
\Omega[1,\gamma,\nu](\delta u,\delta \tau)
= \ & D^2 \phi(\bar{x}(T)) (\delta x(T))^2 + \sum_{j=1}^r \gamma_j D^2 g(\bar{x}(\tau_j)) (\delta x(\bar{\tau}_j))^2 \notag \\
& \qquad + \int_0^T D^2 H^a[t](\delta x(t),\delta u(t))^2  \ \dd t + 2 \int_0^T (\rho_{\bar{\tau}}(t)\cdot \delta \tau) DH[t](\delta x(t),\delta u(t))  \ \dd t.
\end{align}
We finally have the following result.

\begin{thm}
For all $(\delta u,\delta \tau) \in C(\bar{u},\bar{\tau})$, $\Omega[1,\gamma,\nu](\delta u,\delta \tau) \geq 0$.
\end{thm}

\section{Conclusion}
The various transformations that we introduced in this paper enable us to obtain first- and second-order optimality conditions for the time crisis problem over a finite horizon, which presents the particularity to have a discontinuous cost function w.r.t.~the state. 
Since our approach relies in particular on a transverse hypothesis on optimal trajectories, 
further studies could investigate the case when 
optimal trajectories are no longer transverse. As well, we are interested in finding necessary optimality conditions in the case where $T=+\infty$ and 
$\theta(x_0)<+\infty$ (see a first attempt to tackle this case in \cite{bayen3}). Finally, the methodology developed in this paper could be used for numerical simulations of the time crisis problem. 


\begin{thebibliography}{10}

\bibitem{aubin1} {\sc J.-P. Aubin}, {\em Viability Theory}, Systems $\&$ Control: Foundations $\&$ 
Applications.  Birkh\"auser Boston, 1991.

\bibitem{ABSP} {\sc J.-P. Aubin, A.M. Bayen, P. Saint-Pierre}, {\em Viability Theory, New Directions}, Second Editions, Springer, Heidelberg, 2011.

\bibitem{trelat2} {\sc G. Barles, A. Briani, E. Tr\'elat}, {\em Value function for regional problems via dynamic programming and Pontryagin maximum principle}, Math. Control Relat. Fields, vol. 8, 3\&4, pp. 509--533, 2018.   

\bibitem{bayen3} {\sc T. Bayen, A. Rapaport}, {\em Minimal time crisis versus minimum time to reach a viability kernel : a case study in the prey-predator model}, Optimal Control Appl. Methods , \url{https://doi.org/10.1002/oca.2484}
    
\bibitem{bayen2} {\sc T. Bayen, A. Rapaport}, {\em About the minimal time crisis problem}, ESAIM Proc. Surveys , EDP Sci., vol. 57, 
pp. 1--11, 2017.

\bibitem{bayen1} {\sc T. Bayen, A. Rapaport}, {\em About Moreau-Yosida regularization of the minimal time
    crisis problem}, J. Convex Anal. 23 (2016), No. 1, pp. 263--290.

\bibitem{BDP14} {\sc J. F. Bonnans, X. Dupuis, L. Pfeiffer}, {\em Second-order necessary conditions in Pontryagin form for optimal control problems}, SIAM J. Control Optim., Vol. 52, No. 6, 2014, pp. 3887--3916

\bibitem{BDP14b} {\sc J. F. Bonnans, X. Dupuis, L. Pfeiffer}, {\em Second-order sufficient conditions for strong solutions to optimal control problems}, ESAIM Control Optim. Calc. Var., Vol. 159, No. 1, 2014, pp. 1--40

\bibitem{BH09} {\sc J. F. Bonnans, A. Hermant}, {\em Second-order analysis for optimal control problems with
pure state constraints and mixed control-state constraints}, Ann. Inst. H. Poincar\'e Anal. Non Lin\'eaire, 26 (2009), pp. 561--598.

\bibitem{BO10}{\sc J. F. Bonnans and N. P. Osmolovski\v i}, {\em Second-order analysis of optimal control problems
with control and initial-final state constraints}, J. Convex Anal., 17 (2010), pp. 885--913.

\bibitem{bourdin} {\sc L. Bourdin, E. Tr\'elat}, {\em Linear-quadratic optimal sampled-data control problems:
Convergence result and Riccati theory}, Automatica 79 (2017) pp. 273--281. 

\bibitem{clarke2013} {\sc F.H. Clarke}, {\em Functional Analysis, Calculus of Variation, Optimal control}, Graduate Texts in Mathematics, 264, Springer, London, 2013. 

\bibitem{dmitruk1} {\sc A.V. Dmitruk}, {\em The hybrid maximum principle is a consequence of Pontryagin maximum principle}, 
Systems Control Lett. 57 (2008), no. 11, pp. 964--970. 

\bibitem{dmitruk2} {\sc A.V. Dmitruk and A. M. Kaganovich}, {\em Maximum principle for optimal control problems with intermediate constraints}, 
Comput. Math. Model., vol. 22, 2, pp. 180-215, 2011. 

\bibitem{dmitruk3} {\sc A.V. Dmitruk and A. M. Kaganovich}, {\em Quadratic order conditions for an extended weak minimum in optimal control problems with intermediate and mixed constraints}, Discrete Contin. Dyn. Syst. 29 (2011), no. 2, pp. 523--545. 

\bibitem{DSP} {\sc L. Doyen, P. Saint-Pierre}, {\em Scale of viability and minimal time of crisis}, Set-Valued Anal. 5, pp.227--246, 1997.

\bibitem{fatorini} {\sc H.O. Fattorini}, {\em Infinite Dimensional Optimization and Control Theory}, Cambridge University Press, Cambridge,  2013. 

\bibitem{piccoli} {\sc M. Garavello, B. Piccoli}, {\em Hybrid necessary principle}, SIAM J. Control Optim. Vol. 43, 5, pp. 1867--1887, 2005.

\bibitem{trelat} {\sc T. Haberkorn, E. Tr\'elat}, {\em Convergence results for smooth regularizations of hybrid nonlinear optimal control problems}, SIAM . J. Control Optim., vol. 49, 4, pp. 1498--1522, 2011.

\bibitem{Pontry} {\sc L.S. Pontryagin, V.G. Boltyanskiy, R.V. Gamkrelidze, E.F. Mishchenko}, {\em The Mathematical Theory of Optimal Processes}, The Macmillan Co., New York 1964.

\bibitem{vinter} {\sc R. Vinter}, {\em Optimal Control}, Systems and Control: Foundations and Applications, Birkh\"auser, Boston, 2000.

\bibitem{zelikin}{\sc M. I. Zelikin, V. F. Borisov}, {\em Theory of Chattering Control}, 
Systems $\&$ Control: Foundations $\&$ Applications, Birkh\"auser, 1994.

\end{thebibliography}
\end{document}